\documentclass[12pt]{amsart}

\usepackage{graphicx}
\usepackage[electronic]{ifsym}
\usepackage{epstopdf} 

\usepackage{hyperref}

\newtheorem{thm}{Theorem}[section]
\newtheorem{lem}[thm]{Lemma}
\newtheorem{cor}[thm]{Corollary}
\newtheorem{prop}[thm]{Proposition}

\newtheorem{notation}[thm]{Notation}

\newtheorem*{question*}{Question}

\theoremstyle{definition}
\newtheorem*{example*}{Example}
\newtheorem{definition}[thm]{Definition}
\newtheorem*{notation*}{Notation}

\newcommand{\E}{\mathbb{E}}
\newcommand{\N}{\mathbb N}
\newcommand{\R}{\mathbb R}
\newcommand{\Z}{\mathbb Z}
\newcommand{\T}{\mathbb T}
\newcommand{\stair}{{\mbox{\RaisingEdge}}}

\newcommand{\vecv}{\vec v}

\newcommand{\CT}{\mathcal T}

\newcommand{\vol}{\text{vol}}

\begin{document}
\title[Alpern's Lemma for $\mathbb R^d$ actions]
{Rudolph's Two Step Coding Theorem and Alpern's Lemma for $\mathbb R^d$ actions}
\author{Bryna Kra}
\address[Kra]{Department of Mathematics, Northwestern University, Evanston, IL 60208, USA}
\email{kra@math.northwestern.edu}
\author{Anthony Quas}
\address[Quas]{Department of Mathematics and Statistics, University of
Victoria, Victoria BC, V8W 3R4, Canada} \email{aquas@uvic.ca}
\author{Ay\c se \c Sah\.in}
\address[\c Sahin]{Department of Mathematical Sciences, DePaul University, 2320 N. Kenmore Ave,
Chicago, IL 60614, USA} \email{asahin@depaul.edu}
\thanks{The first author was partially supported by NSF grant $1200971$ and the second author 
was partially supported by NSERC}

\dedicatory{Dedicated to the memory of Daniel J. Rudolph}
\maketitle

\begin{abstract}
Rudolph showed that the orbits of any measurable, measure preserving $\mathbb R^d$ action can be measurably tiled by $2^d$ rectangles and asked if this number of tiles is optimal for $d>1$. In this paper, using a tiling of $\mathbb R^d$ by {\it notched cubes}, we show that $d+1$ tiles suffice. Furthermore, using a detailed analysis of the set of invariant measures on tilings of $\mathbb R^2$ by two rectangles, we show that while for $\mathbb R^2$ actions with completely positive entropy this bound is optimal there exist mixing $\mathbb R^2$ actions whose orbits can be tiled by 2 tiles.   
 \end{abstract}

\section{Introduction}

The work in this paper completes a project started by the third author in collaboration with Daniel Rudolph, shortly before his death.  The project started with a question posed by Rudolph about representations of measurable $\mathbb R^d$ actions.  The classical representation theorem for continuous group actions is the Ambrose-Kakutani Theorem \cite{AK} which states that every free, probability measure preserving action of $\mathbb R$ on a Lebesgue space is measurably isomorphic to a flow built under a function.   In \cite{RudolphR1} Rudolph proved his Two Step Coding Theorem showing that given any irrational $\alpha>0$, there is in fact a representation where the ceiling function takes only two values $1$ and $1+\alpha$.   Rudolph proves his result by showing that 
the flow can be factored onto the translation action of $\mathbb R$ on the space of tilings of $\mathbb R$ by two tiles.   
Since a constant ceiling function results in a flow with a non-ergodic time, it is clear that two is the smallest number of tiles for which such a general statement can be true.

In \cite{RudolphRd} and \cite{R2} Rudolph generalized his one dimensional result to show that every free, measure preserving $\mathbb R^d$ action can factored onto the translation action of $\mathbb R^d$ on a space of tilings of $\mathbb R^d$ by $2^d$ rectangles.  His result shows more: in particular the base points of the tiles form a section of the $\mathbb R^d$ action and the constraints on the tilings give rise to a natural $\mathbb Z^2$ action on the return times to the base points of the tiles. Thus Rudolph's theorem gives a representation of the $\mathbb R^d$ action as a suspension flow.

The Rohlin Lemma can be viewed as the discrete analog of the Ambrose-Kakutani Theorem.  In its standard formulation it states that given a free,  measure preserving $\mathbb Z$ action on a Lebesgue probability space $X$, $N\in\mathbb N$, and $\epsilon>0$, the space $X$ can be measurably decomposed into a tower of height $N$ and a set of measure $\epsilon$ usually called the error set.  The union of the base of the tower and the error set are the analogs of the section in the continuous representation theorems, and the Rohlin Lemma can be restated as an orbit tiling result.  
Alpern \cite{Alpern} generalized this result to show that in fact the space can be decomposed into any collection of towers provided the lengths of the towers have no non-trivial common divisor.   In particular, he established the discrete analog of Rudolph's Two Step Coding Theorem  for $\mathbb Z$ actions.  

Alpern's result was generalized to $\mathbb Z^d$ actions by Prikhodko \cite{Prikhodko} and \c Sahin \cite{Ayse} for rectangular towers, but with no restriction on the number of tiles as a function of dimension.  Notably, two tiles suffice in any dimension.  Further, the proof in \cite{Ayse} is a generalization of ideas introduced in \cite{R2}, suggesting that the continuity of the group imposes some restrictions on the number of tiles necessary for a general representation theorem.  Motivated by the $\mathbb Z^d$ Alpern Theorem, Rudolph asked whether the number of tiles in his $\mathbb R^d$ result, $d>1$, was sharp.

In December, 2009  Rudolph and \c Sahin started working on this project with the aim of first exploring the case with $d=2$, and got as far as establishing the first few structural results about tilings of $\mathbb R^2$ with two rectangles (up to Lemma~\ref{lem:aba} in Section~\ref{sec:2-tilings}).   

In this paper we continue this work and develop new techniques to analyze the geometric structure of tilings of $\mathbb R^2$ with two rectangles.  This enables us to  give a complete description of these tilings, as well as the set of probability measures on the space of such tilings invariant under the standard translation action.   We show, in particular, that the standard translation action on this space as a topological dynamical system has entropy zero. 
As a consequence, the analog of Rudolph's statement in two dimensions using two rectangular 
tiles is false. 

In the positive direction we prove that any free, ergodic and measure preserving $\mathbb R^d$ action can be factored onto an $\mathbb R^d$ tiling dynamical system with $d+1$ rectangular tiles.   
The approach we use is essentially that introduced by Rudolph in \cite{R2}.  His construction of the tilings relies on the existence of a periodic tiling by a large rectangle composed of the $2^d$ rectangles of the theorem (referred to as a supertile in our paper) and the ability to move large fundamental domains of the periodic tiling to approximations of arbitrary locations in $\mathbb R^d$.  Moving the fundamental domains relies on performing a series of local perturbations in each coordinate direction and the perturbations in turn require two irrationally related dimensions of the rectangles, thus requiring $2^d$ rectangles. The new ingredient in our work is that we construct a periodic tiling by supertiles which are not rectangular.  In this tiling we are able to move fundamental domains by perturbing in only one direction, and this allows us to reduce the number of rectangular tiles required to $d+1$.  

In what follows we give more precise statements of our results and provide an outline for the paper.  

\subsection{Tiling orbits of $\mathbb R^d$ actions}
In Section~\ref{s:tilingthm} we prove an Alpern Lemma for $\mathbb R^d$ actions by providing a set of tiles with which almost every orbit of any free, ergodic, measure preserving $\mathbb R^d$ action can be tiled.  To state the theorem formally we need to establish some notation.
Given a collection $\mathcal T=\{\tau_1,\tau_2,\ldots,\tau_k\}$ of $d$-dimensional tiles, a {\em $\mathcal T$-tiling} of $\mathbb R^d$ is a covering of $\mathbb R^d$ by translates of tiles from $\mathcal T$ such that they only overlap on boundaries of the tiles.  The set of all tilings of $\mathbb R^d$ by $\mathcal T$ is denoted by $Y_{\mathcal T}$, and the standard translation action of $\mathbb R^d$ on $Y_{\CT}$ is denoted by $S=\{S_{\vec v}\}_{\vec v\in\mathbb R^d}$.   We say that a 
dynamical system $(X,\mu,\{T_{\vec v}\}_{\vec v\in\mathbb R^d}\})$ is {\em $\CT$-tileable} if there exists a factor map 
\begin{equation*}
\phi\colon(X,\mu,\{T_{\vec v}\}_{\vec v\in\mathbb R^d}\})\rightarrow (Y_{\CT},\phi(\mu),\{S_{\vec v}\}_{\vec v\in\mathbb R^d}).
\end{equation*}
 
We show:
\begin{thm}\label{thm:main}
Let $d\ge 1$.  There exists a set $\CT=\{\tau_1,\ldots,\tau_{d+1}\}$ of $d+1$ rectangular tiles in $\mathbb R^d$ such that any free, probability measure preserving $\mathbb R^d$-action $(X,\mu,\{T_{\vec v}\}_{\vec v\in\mathbb R^d})$ is $\CT$-tileable.
\end{thm}

The collection of tiles, called the \emph{basic tiles}, that
satisfy the conclusion of Theorem 1 is obtained by sub-dividing a
single basic tile given by a $d$-dimensional rectangle with a smaller one removed from
one corner. We call this the
supertile and it tiles $\R^d$ periodically. The supertile in
dimension two is an amalgamation of tiles used by Rudolph in
\cite{R2}. The periodic tiling using the
higher-dimensional supertile was previously constructed by Stein
\cite{Stein} and also studied by Kolountzakis
\cite{Kolountzakis}. In these papers, a tiling that is
essentially the periodic tiling that we use (up to scaling in one
coordinate) is called the ``notched cube'' tiling, but 
they do not concern themselves with the decomposition of the notched
cube into the smaller tiles. In keeping with the terminology of the earlier
references, we also refer to the periodic tiling in this work as
the notched cube tiling.

Section~\ref{s:tilingthm} contains the definitions of the basic tiles, the supertile and the notched cube tiling.   In Section \ref{subsec:mixing-notched} we study the topological mixing properties of the notched cube tiling and we prove Theorem~\ref{thm:main} in Section \ref{subsec:mainpf}.

\subsection{Structure of $2$-tiling measures}
In Section~\ref{sec:2-tilings}  we study $\mathcal T$-tilings of $\mathbb R^2$ where $\mathcal T$ consists of two rectangular tiles. The main result of the section is:
\begin{thm}\label{thm:zero}
Let $\CT=\{\tau_1,\tau_2\}$ where the tiles $\tau_i$ are rectangles whose corresponding dimensions are irrationally related.  Then $(Y_\CT,\{S_{\vec v}\}_{\vec v\in\mathbb R^2})$ has zero topological entropy.
\end{thm}
 Since there are obvious dynamical obstructions to proving Theorem~\ref{thm:main} if the corresponding dimensions of the tiles are not irrationally related we obtain as a corollary that three tiles is a sharp result for tileability of $\mathbb R^2$ actions:
\begin{cor}\label{c:sharp2}
There exists a measure preserving $\mathbb R^2$ action that is not $\CT$-tileable for any collection $\CT$ of two rectangular tiles. 
\end{cor}

We conjecture that the analogous result holds for all $d\geq 1$, meaning that there exists 
a measure preserving $\R^d$ action that is not $\CT$-tileable for any collection $\CT$ of $d$ 
rectangular tiles, but our proof does not readily generalize.  

The proof of Theorem~\ref{thm:zero} depends on the fact that the geometric structure of tilings of the plane with tiles from such a collection $\CT$ is quite rigid.  As a consequence of this rigidity, we show that the set of ergodic, invariant measures on $Y_\CT$ can be characterized via a dichotomy:  either the measure is supported on tilings with infinitely many {\em bi-infinite shears} or the measure is supported on tilings with a staircase geometric structure we call {\em staircase tilings}.

Deferring the precise definitions to Section~\ref{sec:2-tilings}, 
shears are horizontal or vertical boundaries between non-aligned tiles.   
The analysis of the infinite shear case is straightforward and is completed in Section~\ref{subsec:patterns-on-shears}.  The detailed  analysis of the structure of staircase tilings and a
complete description of the measures supported on them is the 
content of  Sections~\ref{subsec:shears-in-staircase}--\ref{subsec:measures-on-staircases}.
In the absence of infinite  shears  we identify a 
particular feature of the tilings, namely a pair of families of \emph{staircases} constructed from terminating shears. These tilings 
have a $\Z^2$ lattice type structure
that allows us to 
represent a measure-preserving $\R^2$ action on a staircase tiling
as a suspension over a $\Z^2$ action.   

In Section~\ref{subsec:entropy}, we compute the entropy of invariant measures on 
$Y_{\mathcal T}$ and use this to show that there exist $\R^2$ systems that 
are not tileable with only two tiles.  
On the other hand, in Section~\ref{sec:positive-2-tiles}, we show that 
some classes of measure preserving actions are tileable with exactly 
two tiles.  
Finally, in Section~\ref{sec:mixing} we show that for any $\mathcal T$ consisting of a pair of rectangles whose corresponding dimensions are incommensurable, 
there exists a mixing $\R^2$ measure preserving system that 
is $\mathcal T$-tileable. 

We are grateful to Valery Ryzhikov for bringing his work in ~\cite{Ryz1} and ~\cite{Ryz2} to our attention.  In ~\cite{Ryz1} he proves that there is  no ``epsilon-free'' Rohlin Lemma  for $\mathbb Z^2$-actions of completely positive entropy by proving that the factor of the $\mathbb Z^2$ action obtained from the partition consisting of an epsilon-free tower and its error set has to have zero entropy.  The proof of this claim is strikingly similar to our argument in Section~\ref{subsec:entropy}.  It is also interesting to note that in spite of the similarities in their proofs, neither result can be obtained from the other.  
 
%
%
%

\subsection{Acknowledgment} We thank Jason Siefken and Tim Austin
for helpful conversations.

\section{The Notched Cube Tiling and the Proof of Theorem~\ref{thm:main}}\label{s:tilingthm}


\subsection{The notched cube tiling $\mathcal N$}
The $d+1$ basic tiles that satisfy Theorem~\ref{thm:main} are defined as follows:
\begin{definition}
\label{def:T-alpha}
Fix an integer $d\in\mathbb N$ and $\alpha$ satisfying $0< \alpha< 1$.  
We define an associated set $\mathcal T_\alpha=\{\tau_1,\ldots,\tau_{d+1}\}$ of 
basic tiles, where
\begin{align*}
\tau_i&=[0,\alpha)^{i-1}\times[0,1-\alpha)\times[0,1)^{d-i}\qquad\text{for $i<d$}\\
\tau_d&=[0,\alpha)^{d-1}\times[0,1-\alpha)\qquad\text{and}\\
\tau_{d+1}&=[0,\alpha)^d.
\end{align*}
%
\end{definition}

Note that the unit cube can be tiled by the basic tiles:
\begin{equation*}
[0,1)^d=\left(\bigcup_{i=1}^{d+1}\tau_i+(\overbrace{1-\alpha,\ldots,{1-\alpha}}^{i-1},0,\ldots,0)\right).
\end{equation*}
In particular if the $i$th coordinate of a point $x\in[0,1)^d$ is the first coordinate in the range $[0,1-\alpha)$, then it lies in $\tau_i$.  Otherwise it lies in the unique translate of $\tau_{d+1}$ that occurs in its tiling.
See Figure~\ref{fig:cubedecomp} for a picture of the four tiles used in three dimensions tiling the 
unit cube in $\R^3$.

\begin{figure}[h!]
\includegraphics[width=1.5in]{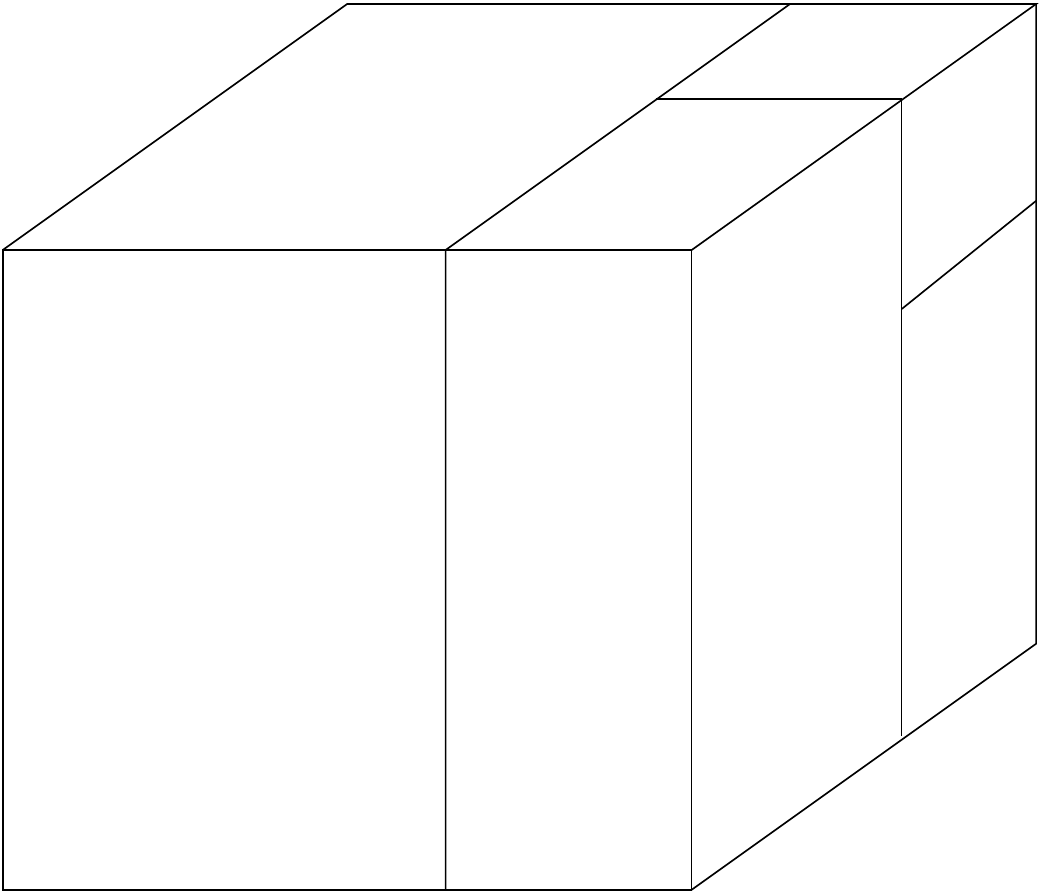}
\caption{The unit cube decomposed into the $d+1$ basic tiles}
\label{fig:cubedecomp}
\end{figure}

\begin{figure}[h!]
\includegraphics[width=1.5in]{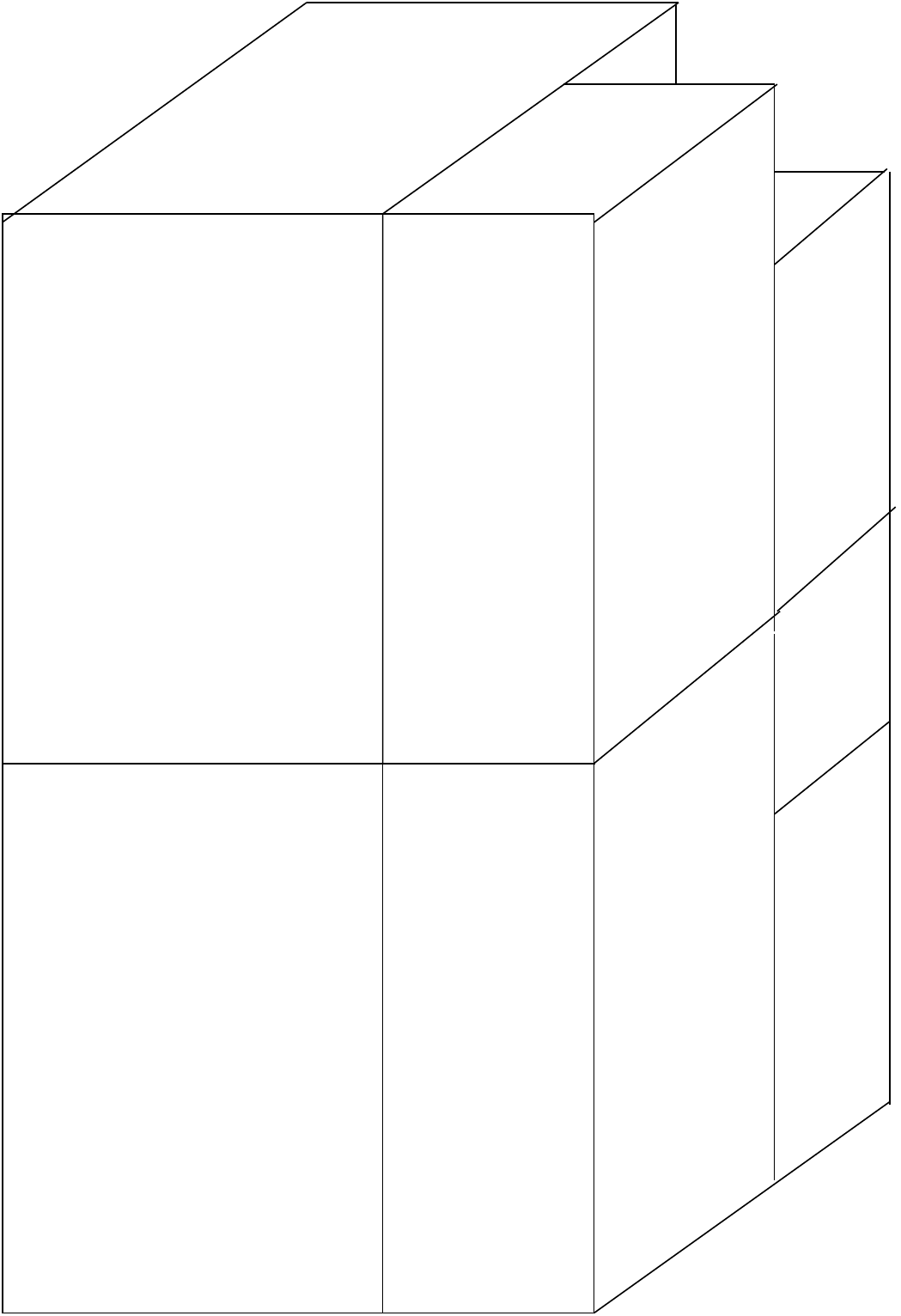}
\caption{The supertile $\tau^*_\alpha$ and its tiling by the basic tiles}
\label{fig:fd}
\end{figure}

The supertile consists of two stacked copies of the unit cube, with a copy of $\tau_{d+1}$ removed from a corner (see Figure~\ref{fig:fd}):
\begin{definition}
For $0<\alpha<1$ and associated $\mathcal T_\alpha$, the  {\it supertile} $\tau^*_\alpha$ for
 $\mathcal T_\alpha$ is defined to be
\begin{align*}
\tau_\alpha^*&=[0,1)^d\cup \bigcup_{i=1}^d\left(\tau_i+(\overbrace{1-\alpha,\ldots,{1-\alpha}}^{i-1},0,\ldots,0,1)\right)\\
&=\left([0,1)^d\times[0,2)\right)\setminus\bigl([1-\alpha,1)^{d-1}\times[2-\alpha,2)\bigr).
\end{align*}

We let $\mathcal T_\alpha^* = \mathcal T_\alpha\cup\tau^*_\alpha$ denote the 
union of the basic tiles and the supertile.
\end{definition}


\begin{notation*}
In what follows we fix an $\alpha\in(0,1)$, its associated sets of tiles $\mathcal T_\alpha$ 
and $\mathcal T_\alpha^*$, and the supertile $\tau^*_\alpha$.  For notational convenience we suppress the subscripts $\alpha$ for the remainder of the section and simply write $\mathcal T$, 
$\mathcal T^*$ and $\tau^*$.
\end{notation*}

Since the supertile can be decomposed into basic tiles,  we immediately have:
\begin{lem}\label{decomp}
Given any element $y^*\in Y_{\CT^*}$, there is a corresponding tiling $y\in Y_{\CT}$ obtained by decomposing each translate of the supertile into its basic tile components.
\end{lem}

We are left with showing that $Y_{\CT^*}$  is not empty. 
 \begin{prop}\label{p:periodic}
Let $A$ be the matrix with non-zero entries
  $A_{i,i}=1$ for $i<d$; $A_{i+1,i}=-\alpha$ for $i<d$;
$A_{i,d}=1-\alpha$ for $i<d$ and $A_{d,d}=2-\alpha$:
\begin{equation}\label{allA}
A=
\begin{pmatrix}
1&0&\ldots &&0&1-\alpha\\
-\alpha&1&0&\ldots&0&1-\alpha\\
0&-\alpha&1&0&\ldots&1-\alpha\\
\\
\vdots&&\ddots&\ddots&&\vdots\\
\\
0&\ldots&0&-\alpha&1&1-\alpha\\
0&\ldots&&0&-\alpha&2-\alpha
\end{pmatrix}.
\end{equation}
Then $\tau^*$ is a fundamental domain for the translation action of $A\Z^d$ on
  $\R^d$.
\end{prop}

Before proving the proposition, we give a characterization of the fundamental domain of a translation action on $\mathbb R^d$:
\begin{lem}\label{lem:fd}
  Fix an integer $d\geq 1$.  Let $A$ be an invertible $d\times d$ matrix with entries in $\mathbb R$ and let $B\subset\R^d$ 
  be a bounded, measurable set satisfying $B\cap (B+A\vecv)=\emptyset$ for every $\vecv\in\Z^d$.
  Then the volume $\vol(B)\le \det(A)$. 
 
  Furthermore, if $\vol(B) = \det(A)$,  then $B$
  agrees up to a set of measure 0 with a fundamental domain for the
  action of $A\Z^d$ on $\R^d$ by translation.
\end{lem}
\begin{proof}
Set $Q=\R^d/A\Z^d$. Consider the projection map $\pi\colon\R^d\to Q$
given by $\pi(x)=x+A\Z^d$. Since $A[0,1)^d$ is a fundamental domain for the 
action of $A\Z^d$ on $\R^d$ by translation and
has volume $\det A$, the natural measure $m$ induced on the
quotient has total mass $\det A$. By assumption, $\pi$ is a bijection between
$B$ and $\pi(B)$, meaning that $\vol(B)=m(\pi(B))\le \det A$.

If equality holds, then setting  $N$ to be the null set $\pi^{-1}(Q\setminus
\pi(B))\cap A[0,1)^d$, we have that  $B\cup N$ is a fundamental domain for the
action of $A\Z^d$ on $\R^d$ by translation.
\end{proof}

We use this to complete the proof of the proposition: 
\begin{proof}[Proof of Proposition~\ref{p:periodic}]
Expanding across the bottom row and using induction, one 
can check that $\det A=2-\alpha^d$.

By Lemma \ref{lem:fd}, it suffices to show that $\left(\tau^*+A\vecv\right)\cap
\tau^*=\emptyset$ for $\vecv\in\Z^d\setminus\{0\}$.  
By symmetry, it suffices to
establish this for $\vecv$ satisfying $v_d\ge 0$.  

We prove the disjointness case by case:
\begin{description}
\item[Case 1]
($v_d=0$). Let $v_i$ be the first non-zero component of $\vec v$.  By symmetry it suffices to assume $v_i>0$.   and so $(A\vecv)_i\ge 1$. Since the $i$ coordinates
of $\tau^*$ lie in $[0,1)$, it follows that $\tau^*$ and $\tau^*+A\vec v$ are
disjoint.

\item[Case 2]
($v_d>0$; $v_i>0$ for some $i<d$). Let $i$ be the first
coordinate such that $v_i>0$.  Since $(A\vecv)_i=-v_{i-1}\alpha + v_i+v_d(1-\alpha)$,
where we set $v_{-1}=0$, $(A\vecv)_i$ is at least 1 and we are done
as in Case 1.

\item[Case 3]
($v_d>0$; $v_{i}<0$ for some $i<d$). We can assume that $i$ is the
greatest index up to $d-1$ for which $v_i<0$. Then as above we have
$(A\vecv)_{i+1}\ge1$ and the disjointness follows. 

\item[Case 4]
($v_d>0$; $v_i=0$ for $1\le i<d$).  Then $\tau^*$ is contained in the union of the half
spaces $\bigcup_{i< d}\{x\colon x_i<1-\alpha\} \cup\{x\colon
x_d<2-\alpha\}$, whereas $\tau^*+A\vecv$ is contained in the intersection of the half
spaces $\bigcap_{i<d}\{x\colon x_i\ge 1-\alpha\}\cap \{x\colon x_d\ge
2-\alpha\}$, which are complementary sets.  
\end{description}
\end{proof}

Translating the statement of Proposition~\ref{p:periodic} into tilings, we have: 
\begin{cor}\label{cor:periodic}
There is a periodic tiling of $\mathbb R^d$ by the supertile $\tau^*$.
\end{cor}
Let $\mathcal N$ denote the periodic tiling of $\mathbb R^d$ by the supertile given by Corollary~\ref{cor:periodic}. We call $\mathcal N$ the {\it notched cube tiling}.  

\subsection{$\mathcal N$ is a uniform filling set}
\label{subsec:mixing-notched}
 
As before, we fix $0<\alpha<1$ and the associated tiles and tilings.  In the proof of Theorem~\ref{thm:main} the tiling $\mathcal N$ plays the role of a {\it uniform filling set} as introduced in \cite{RS4}.  In particular, we use $\mathcal N$ as a canvas, using a sequence of Rohlin towers to inductively tile larger and larger pieces of orbits.  The filling property of $\mathcal N$ is used to achieve agreement between tilings from different stages and thus is key to ensuring that the sequence of tilings converges.  The difference in the filling property defined below in Proposition~\ref{p:shift} comes from the fact that we are using a continuous group action.  

The filling properties of $\mathcal N$ are a consequence of the way faces of the translate of the supertile intersect.  We call the faces of $\tau^*$ with extremal $\vec e_d$ coordinates the {\it top} and {\it bottom} faces, where $\vec e_d$ denotes the $d$th standard basis element of $\mathbb R^d$.

\begin{lem}\label{lem:intersect}{\em(Tile Intersection Property)} For any $0<\alpha<1$, let $\tau^*$ be the supertile for the set of tiles $\CT$.  Then for any $\vec v\in\mathbb Z^d$, if the top boundary of the translate $\tau^*+ A\vec v$ intersects the bottom boundary of any other translate of $\tau^*$, then this intersection exactly agrees with the top boundary of a basic tile of $\tau^*+A\vec v$.
\end{lem}

\begin{proof} It suffices to consider the case that $\vec v=0$.
 For $1\le j\le d$, let 
 $$\vec u_j=(\overbrace{0,\ldots,0}^{j-1},{-1},-1,\ldots,-1,1)^T.$$
Then by direct calculation, we have 
\begin{equation}\label{eq:Auj}
\begin{split}
A\vec u_j&=(\overbrace{1-\alpha,\ldots,1-\alpha}^{j-1},{-\alpha},0,\ldots,0,2)^T\text{ for $j<d$; and}\\
A\vec u_d&=(1-\alpha,\ldots,1-\alpha,2-\alpha)^T.
\end{split}
\end{equation}

The top of $\tau^*$ is 
$$
\big(\big([0,1)^{d-1}\setminus[1-\alpha,1)^{d-1}\big)\times\{2\}\big)\cup\big([1-\alpha,1)^{d-1}\times \{2-\alpha\}\big).
$$
Notice that by \eqref{eq:Auj}, for $j<d$, $\tau^*+A\vec u_j$ has base 
$$
[1-\alpha,2-\alpha)^{j-1}\times [-\alpha,1-\alpha)\times [0,1)^{d-j-1}\times \{2\}.
$$
Intersecting these with the top of $\tau^*$, we obtain
$$
[1-\alpha,1)^{j-1}\times [0,1-\alpha)\times [0,1)^{d-j-1}\times\{2\},
$$
which is exactly the top of $\tau_j+(\overbrace{1-\alpha,\ldots, {1-\alpha}}^{j-1},0,\ldots,0,1)^T$,
one of the component basic tiles forming the supertile.

The base of $\tau^*+A\vec u_d$ is 
$[1-\alpha,2-\alpha)^{d-1}\times\{2-\alpha\}$. The intersection with the top of $\tau^*$ is
$[1-\alpha,1)^{d-1}\times\{2-\alpha\}$, which is the top of $\tau_d+(1-\alpha,\ldots,1-\alpha,1)^T$.

Hence we have found a collection of $d$ translates of the supertile at the origin, 
whose bases intersect the supertile at the origin exactly in the tops of all the $d$ basic tiles
forming the top of  the supertile and thus the $\tau^*+A\vec u_j$ are the only translates of the supertile whose bottoms can intersect the top of $\tau^*$.
 
 \end{proof}

\begin{definition}
Let $R\subset\mathbb R^d$, $\CT$ be a collection of $d$-dimensional tiles, and $y\in Y_\CT$.  The restriction of the tiling given by $y$ to the region $R$ is denoted by $y[R]$ and is called a {\it patch}.  The patch $\mathcal N[R]$  is called a {\it grid patch}, and for $y\in Y_{\CT^*}$, if $y[R]=\mathcal N[R]$, then $y[R]$ is said to be {\it grid tiled}. 

If for $\vec v\in\mathbb R^d$, $y[R]=\left(S_{\vec v}(y')\right)[R]$, then we say that $y[R]=y'[R+\vec v]$.
\end{definition}

\begin{definition}
Given a finite subset $S\subset\mathbb Z^d$ and the matrix $A$ defined 
in~\eqref{allA},  define  $R(S)\subset\mathbb R^d$ by 
\begin{equation*}
R(S)=\bigcup_{\vec v\in S}(\tau^*+A\vec v).
\end{equation*}

For $K,M\in\mathbb N$, set
 \begin{equation}\label{special}
R_{K,M}=R(\{\vec v\in \mathbb Z^d\colon -K\le v_i<M+K\}) 
\end{equation}
and denote $R_{0,M}=R_M$. 

\end{definition}

We use the Tile Intersection Property to shift any grid patch in $\mathcal N$ down one unit and tile $\mathbb R^d$ while only changing $\mathcal N$  slightly:
\begin{prop}\label{p:shift}
For any $M\in\mathbb N$ and $\vec w\in A\mathbb Z^d$, 
 there exists a tiling $y\in Y_{\mathcal T^*}$ such that  
\begin{equation*}
y[(R_{1,M}+\vec w)^c]=\mathcal N[(R_{1,M}+\vec w)^c] \text{ and  }y[R_M+\vec w-\vec e_d]= \mathcal N[R_{M}+\vec w].
\end{equation*}
\end{prop}

\begin{proof}
Fix $M$ and $\vec w$.  It suffices to show that if the patches $\left(R_{1,M}+\vec w\right)^c$ and $R_{M}+\vec w-\vec e_d$ are both grid tiled, then the region $\left(R_{1,M}+\vec w\right)\setminus \left(R_M+\vec w-\vec e_d\right)$ can be tiled using tiles of $\CT^*$.  By the Tile Intersection Property (Lemma~\ref{lem:intersect}) we know that if if $\tau^*+A\vec v$ lies in the set $(R_{1,M}+\vec w) \setminus (R_M+\vec w)$ and $\tau^*+A\vec v$ intersects $R_M+\vec w-\vec e_d$, then the part of $\tau^*+A\vec v$ that is covered by more than one tile consists of a complete union of basic tiles.
The bottom of the untiled region can thus can be tiled by decomposing these  $\tau^*+A\vec v$ in $\mathcal N$ into their basic tile components and removing the basic tiles that lie in the intersection. The top of the untiled region can be decomposed into $d$-dimensional rectangles with their $d$th dimension $1$ and the remaining dimensions corresponding to tops of basic tiles $\tau_1,\ldots,\tau_d$.  The tiling $y$ is completed by tiling each such rectangle either by $\tau_1,\ldots,\tau_{d-1}$ or by the pair $\tau_d,\tau_{d+1}$ stacked on top of one another.

%
 \end{proof}
By restricting our choice of $\alpha$, 
the ability to move a grid patch down in the $\vec e_d$ direction suffices to guarantee that there is a tiling which sees a grid patch arbitrarily close to any desired location:

\begin{cor}\label{c:move} {\em ($\mathcal N$ is a uniform filling set)} Suppose $\alpha$ is an irrational which is not algebraic of degree $d$ or lower.  Then given $\epsilon>0$, there exists $K\in\mathbb N$ such that for any $\vec v\in\mathbb R^d$  and any  $M$, there exists $\vec u\in\mathbb R^d$ with 
$\Vert\vec v-\vec u\Vert<\epsilon$ and a tiling $y\in Y_{\mathcal T^*}$ such that 
\begin{equation*}
 y[R_{M}+\vec u]= \mathcal N[R_M] \text{ and } y[\left(R_{2K,M}+\vec u\right)^c]=\mathcal N[\left(R_{2K,M}+\vec u\right)^c].
\end{equation*}
 \end{cor}
\begin{proof}
Fix $\epsilon>0$. The condition on $\alpha$ guarantees that $A^{-1}\vec e_d$
satisfies $\vec n\cdot A^{-1}\vec e_d$
is irrational for all non-zero $\vec n\in\Z^d$, and so  multiples
of $A^{-1}\vec e_d$ are dense in $\R^d/\Z^d$.
By compactness of $\tau^*$, there exists $K>0$ depending only
on $\epsilon$ such that any vector in $\R^d$ can be approximated within
$\epsilon$ (mod $A\Z^d$) by $-k\vec e_d$ for some $k$ with $0\le k\le K$.

Choosing such $k$ with  $0 \leq k\le K$ and $-k\vec e_d$ approximating $\vec v$ within
$\epsilon$ (mod $A\Z^d$), we have that there exists $\vec w\in A\mathbb Z^d$ 
satisfying $\Vert\vec v-\vec u\Vert<\epsilon$, where  $\vec u=\vec w-k\vec e_d$.
Applying Proposition~\ref{p:shift} inductively, there exists a sequence of
tilings $(y^{(j)})_{1\le j\le k}$ such that
\begin{align*}
y^{(j)}[\vec w+R_M-j\vec e_d] &=  \mathcal N[\vec
w+R_M-j\vec e_d]\text{; and}\\
y^{(j)}[\vec w+R_{j,M}^c]&=\mathcal N[\vec w+R_{j,M}^c].
\end{align*}

In particular, $y=y^{(k)}$ satisfies $y[\vec
u+R_M]= \mathcal N[\vec u+R_M]$.
An immediate 
consequence of the Tile Intersection Property 
(Lemma~\ref{lem:intersect}) is that  for any $j,M\in\mathbb N$,
$$
R_{j+1,M}
\supset \bigcup_{t\in[0,1]}(R_{j,M}+t\vec e_d).
$$
Thus $R_{2K,M}\supset
\bigcup_{t\in[0,1]}(R_{K,M}+t\vec e_d)$ and so 
$\vec w+R^c_{K,M}\supset \vec u+R^c_{2K,M}$. It follows that
$y[\vec u+R^c_{2K,M}]=\mathcal N[\vec u+R^c_{2K,M}]$, as required.
%
\end{proof}

To prove Theorem~\ref{thm:main}, we need two generalizations of  Corollary~\ref{c:move}.  First, we shift patches that are not perfect grid patches but are surrounded by a sufficiently large layer of supertiles.  Further, we shift two patches independently  provided they and their supertile collars do not intersect.

\begin{definition}
Let $y\in Y_{\mathcal T^*}$ and $K,M\in\mathbb N$.  We say the patch $y[R_M]$ has a {\em grid collar of size $K$} if
$y[R_{K,M}\setminus R_M]$ is grid tiled.  
\end{definition}

\begin{cor}\label{c:mov2}
Suppose $\alpha$ is an irrational which is not algebraic of degree $d$ or lower.   For all $\epsilon>0$ there exists $K\in\mathbb N$ such that for any finite collection $\vec v_1,\ldots,\vec v_k\in\mathbb R^d$,  $M_1,\ldots,M_k\in\mathbb N$, $y_i\in Y_{\mathcal T^*}$ where
 \begin{equation}\label{eq:disjnt}
 \text{the sets }R_{2K+1,M_i}+\vec v_i \text{ are pairwise disjoint}
 \end{equation}
and the patches $y_i[R_{M_i}]$ have a grid collar of size at least $2K$, there exist $\vec u_1,\ldots, \vec u_k\in\mathbb R^d$ with
\begin{equation}\label{e:small}
\Vert\vec u_i-\vec v_i\Vert<\epsilon\qquad\text{for $i=1,\ldots, k$}
\end{equation}
and a tiling $y\in Y_{\mathcal T^*}$ with 
\begin{equation*}
y\left[R_{M_i } +\vec u_i\right]= y_i\left[R_{M_i }\right] 
\end{equation*}
and
\begin{equation*}
y\left[\left(\bigcup_{i=1}^kR_{K,M_i}+\vec u_i\right)^c\right]=\mathcal N\left[\left(\bigcup_{i=1}^kR _{K,M_i}+\vec u_i\right)^c\right]
\end{equation*}
for all $i=1,\ldots,k$.
\end{cor}
\begin{proof}
Let $\vec w_i\in A\mathbb Z^d$ be such that $\vec v_i\in \tau^*+\vec w_i$.   By \eqref{eq:disjnt} we can apply the proof of Corollary~\ref{c:move} independently  to the patches $\mathcal N[R_{M_i}+\vec w_i]$ in $\mathcal N$.  We complete the proof by tiling each  $R_{M_i}+\vec u$ by   $y_i[R_{M_i}]$.
\end{proof}

\subsection{Proof of Theorem~\ref{thm:main}}\label{subsec:mainpf}
The strategy that we follow in this section is a minor variation of strategies that have appeared in works of many authors, notably Rudolph \cite{RudolphRd},\cite{R2} (in which rectangular tilings of $\mathbb R^d$ appeared) and Robinson and \c Sahin \cite{RS4} (in which uniform filling sets were introduced).  

The tiling is constructed using a sequence of Rohlin towers.  For ease of notation we abbreviate an $\mathbb R^d$ action $\{T_{\vec v}\}_{\vec v\in\mathbb R^d}$ by $T$.

\begin{definition}
Let $(X,\mu,T)$ be a measurable, free, and measure preserving $\mathbb R^d$ action on a Lebesgue probability space. Let $R\subset \mathbb R^d$ and let $B\subset X$ be measurable with $\mu(B)>0$ such that $T_{\vec v_1} B\cap T_{\vec v_2}B=\emptyset$ for all distinct $\vec v_1,\vec v_2\in R$.  Then the set $\bigcup_{\vec v\in R}T_{\vec v}(B)$ is called the {\em Rohlin tower} for $T$ of {\em shape} $R$, with {\em base} $B$.  The set $E=X\setminus \bigcup_{\vec v\in R}T_{\vec v}(B)$ is called the {\em error set}, and $\mu(E)$ is called the {\em error} of the tower.  For $B'\subset B$, the set $\bigcup_{\vec v\in R}T_{\vec v}(B')$ is called the {\it slice of the tower based at $B'$}.
\end{definition}

Lind extended the classical Rohlin Lemma to $\mathbb R^d$ actions \cite{Lind} , showing that given any size $d$-dimensional rectangle in $\mathbb R^d$, any $\epsilon>0$, and any free, measure preserving $\mathbb R^d$ action there is a Rohlin tower of that shape with error $\epsilon$.  The following result can be obtained from Lind's $\mathbb R^d$ Rohlin Lemma, or from Ornstein and Weiss's generalization of the Rohlin Lemma to amenable group actions \cite{OrnsteinWeiss}: 
 \begin{lem}\label{lem:tower}
Let $(X,\mu,T)$ be a measurable, free, and measure preserving $\mathbb R^d$ action on a Lebesgue probability space.   Let  $A$ be the matrix defined in (\ref{allA}) and $R_M\subset\mathbb Z^d$ be as defined in  (\ref{special}).  
Then for all $\epsilon>0$ and $M\in\mathbb N$, there exists a Rohlin tower for $T$ of shape $R_M$ with error at most $\epsilon$.
\end{lem}

\begin{proof}{\em (Of Theorem~\ref{thm:main})}
 Fix an $\alpha$ satisfying the 
hypotheses of Corollary~\ref{c:move} and a measure 
preserving system $(X,\mu,T)$ that satisfies the hypotheses of the theorem.  We prove that $T$ is $\mathcal T_\alpha^*$ tileable.  The result then follows from Lemma~\ref{decomp}.

We give the proof by inductively constructing a sequence of Rohlin towers and a sequence of measurable maps $\Phi_n$ mapping elements of the tower to elements of  $Y_{\mathcal T_\alpha^*}$.  Let $K_n$ be an increasing sequence obtained by applying Corollary~\ref{c:mov2} with $\epsilon=2^{-n}$ for each $n$.   Let $M_n$ be an increasing sequence with $M_n\rightarrow\infty$ and $\epsilon_n$ a decreasing sequence with $\epsilon_n\rightarrow 0$ be chosen such that 
\begin{equation}\label{e:fast}
\frac{\vol\left(R_{M_n}\right)}{\vol\left(R_{2K_n+2K_{n-1}+M_{n-1},M_n}\right)}(1-\epsilon_n)>1-2^n.
\end{equation}
Construct a sequence of Rohlin towers with base sets $B_n$ of shape $R_{2K_n,M_n}$ and error $\epsilon_n$.  Then by \eqref{e:fast} 
\begin{equation}\label{eq:Jn}
\mu\left(\bigcup_{\vec v\in J_n} T_{\vec v}B_n\right)>1-2^{-n}
\end{equation}
where $J_n=R_{-(2K_{n-1}+M_{n-1}),M_n} $ denotes those elements $\vec v\in R_{M_n}$ such that an entire ball whose diameter equals that of $R_{2K_{n-1},M_{n-1}}$ lies in $R_{M_n}$.  
 
If \eqref{eq:Jn} holds, then by the easy direction of the Borel-Cantelli Lemma we have that
\begin{equation}\label{BC}
\text{$\mu$-a.e. $x$ belongs to $\bigcup_{\vec v\in J_n}T_{\vec v}B_n$ for all but finitely many $n$}.
\end{equation}

The maps $\Phi_n$ are constructed such that if $x\in B_n$  then the following conditions hold:

\begin{gather}
\text{$\Phi_n(x)[R_{K_n,M_n}]$  has a grid collar of size $K$,}\label{g:tile}\\
\text{for $\vec u\in R_{M_n}$,  $\Phi_n(T^{\vec u}x)=S^{\vec u}(\Phi_n(x))$,}\label{g:factor}
\end{gather}
and finally if the slice of the Rohlin tower based at $x$ lies in $C_{n+1}$, the core of the $(n+1)$ stage tower, then both $\Phi_n$ and $\Phi_{n+1}$ tile this slice and we require that these tilings agree up to a small translation.  More formally, suppose $x=T^{\vec u} y$ for some $y\in B_{n+1}$ and $\vec u\in J_{n+1}$.   Then
 \begin{equation}\label{e:agree}
\Phi_n(x)[{R_{M_n}}]\text{ and } \Phi_{n+1}(y)[{R_{M_n}+\vec u}]\text{ agree up to a translate of $2^{-n}$}.
\end{equation}

Assuming that these properties are established, we finish the proof by observing that by \eqref{BC},\eqref{g:tile}, and \eqref{e:agree} for $\mu$-almost every $x$ the sequence $\Phi_n(x)$ converges to a tiling $\Phi(x)\in Y_{\mathcal T_\alpha^*}$.   By \eqref{g:factor}, $\Phi$ is a factor map from $T$ onto the translation action $S$ on $Y_{\mathcal T_\alpha^*}$.  It therefore suffices to show we can build a sequence of maps satisfying these conditions.

For $x\in B_1$ we set $\Phi_1(x)=\mathcal N$ and we extend the definition to $y=T^{\vec v}x$ for $\vec v\in R_{M_1}$ by setting $\Phi_1(y)=S^{\vec v}(\Phi(x))$ satisfying (\ref{g:tile}) and (\ref{g:factor}) with $n=1$.

Now suppose that a sequence of maps satisfying \eqref{g:tile},\eqref{g:factor}, and \eqref{e:agree} have been defined up to stage $n$ and fix $x\in B_{n+1}$. To define $\Phi_{n+1}(x)$ we let $\vec v_i\in R_{M_{n+1}}$ denote those locations for which $T_{\vec v_i}x\in B_n$.  Since these are slices of towers stage $n$ towers, the regions $R_{2K_n+1,M_n}+\vec v_i$ satisfy the conditions of Corollary~\ref{c:mov2} with $y_i=\Phi_n(x)[R_{M_{n}}]$.  The tiling $\Phi_{n+1}(x)$ is then given by Corollary~\ref{c:mov2} and \eqref{e:small} yields that \eqref{e:agree} is satisfied at this stage.
\end{proof}

\section{The structure of 2-tilings}
\label{sec:2-tilings}

\subsection{Invariant measures on tilings of $\R^2$: basic observations}

In this section, we consider tilings of $\mathbb R^2$ by rectangles.  Let $\CT$ be a finite collection of rectangular tiles and for $A\in\CT$
let $w_A$ denote the width of $A$ and $h_A$ denote its height.  
Once $\mathcal T$ is fixed, we omit it from the notation 
and refer to $Y_{\mathcal T}$ as $Y$.
By an 
invariant measure on $Y$, we mean a measure invariant under the standard action of $\mathbb R^2$ on $Y$ by translation.

We start with some general properties of invariant measures 
on $Y_{\mathcal T}$, and
then specialize to the two tile case where the widths and heights of the two tiles are rationally independent. 

\begin{definition}
Two tiles in a tiling $y\in Y$ are said to meet \emph{full face to full face}
if a face of one tile exactly agrees with a face of the other.
In this case, the common face is called a {\em shared edge} of the tiling.
A \emph{shear} is a maximal line segment (horizontal or vertical) consisting entirely of
faces of tiles that contains no shared edges. 
A shear is \emph{semi-infinite} if it is a half line and \emph{bi-infinite} if it is a full line.
\end{definition}

An invariant measure almost surely rules out semi-infinite shears:
\begin{lem}\label{lem:nosemiinf}
	Let $\mu$ be an invariant probability measure on $Y$.
	Then $\mu$-almost every element of $Y$ contains no semi-infinite shear.
\end{lem}

This lemma is essentially the Poincar\'e recurrence theorem.  
As usual, we let 
$\vec e_1, \ldots, \vec e_d$ denote the standard basis elements of $\mathbb R^d$.  

\begin{proof}
	It suffices to show that $\mu$-almost every element of $Y$ contains no semi-infinite horizontal
	shear that is infinite to the right (the other directions being established identically).
	Let $r<\min_{A\in\mathcal T} h_A$ and let $E$ be the set of tilings such that a semi-infinite right-pointing
	horizontal shear has its left endpoint in $[0,1)\times[0,r)$. 

	The sets $S_{n\vec e_1}E$ are disjoint, as if a tiling lies in two of these sets, 
	then it has distinct semi-infinite rightward shears  (necessarily at different heights)
	with a vertical separation that is less than $r$. This is impossible since all tiles have height exceeding $r$.
	Since the sets have the same measure, they must all be of measure 0. We conclude the proof 
	by taking a union of countably many translates of $E$.
	
\end{proof}

We say that an invariant measure $\mu$ on $Y$ is {\em ergodic in direction $\vec v\in\R^2$} if 
$\mu$ is ergodic with respect to the action of $(S_{t\vec v}\colon t\in\R)$. 
A measure that is ergodic in the coordinate directions
${\vec e_1}$ and ${\vec e_2}$ has no infinite shears:

\begin{lem}\label{lem:nobiinf}
	Let $\mu$ be an invariant probability measure on $Y$. If $\mu$ is 
	ergodic in the two coordinate directions, then $\mu$-almost every element
	of $Y$ contains no infinite shears.
\end{lem}

\begin{proof}
	We have already ruled out semi-infinite shears so it remains to rule out bi-infinite shears.
	By symmetry, it suffices to exclude bi-infinite horizontal shears. 
	Let $r<\min_{A\in\mathcal T}h_A$. The subset $W$ of $Y$ consisting of points such that a bi-infinite 
	horizontal shear enters $[0,r)^2$ is disjoint from $S_{r\vec e_2}W$, but has the same measure,
	which must therefore be less than 1. Since $W$ is invariant under the horizontal action, it follows that $W$
	has measure 0. Taking countably many translates shows the set of points 
	with bi-infinite horizontal shears has measure 0.
\end{proof}

\subsection{Independent heights and widths}
For the remainder of this section, we specialize
to the case that $\mathcal T$ consists of two basic tiles, $A$ and $B$ of dimensions
$w_A\times h_A$ and $w_B\times h_B$ respectively.  We write $Y_{A,B}$ for the 
collecting of tilings using $A$ and $B$, and as before omit the subscripts on $Y$ 
when the context is clear.  

\begin{lem}\label{lem:incommens}
Suppose that $A$ and $B$ are tiles such that $Y_{A,B}$ supports a measure that is ergodic
in the two coordinate directions. Then the heights of $A$ and $B$ are incommensurable, as are the widths of the
tiles.
\end{lem}

\begin{proof}
	Suppose for a contradiction that the widths of $A$ and $B$ are rationally related, so that
	$w_A=p\alpha$ and $w_B=q\alpha$ with $p,q\in\N$. 
	For this proof, we refer to the left and right boundaries of a tile $T$
	as the vertical boundary $\partial_v(T)$. 
	Fix a tiling $z\in Y$. We say that two tiles, $T$ and $T'$ in $z$ are adjacent if $\partial_v(T)\cap
	\partial_v(T')\ne\emptyset$. (Notice that this could happen in many ways: the left boundary of $T'$ could overlap
	with the right boundary of $T$; the lower left corner of $T'$ could be the
	upper left corner of $T$, etc.)
	We let $\sim$ be the transitive closure of the adjacency relation. If two tiles
	are related, then the $x$-coordinates of their left edges differ by an integer multiple
	of $\alpha$. For a tile $T$, let $\pi_2(T)$	be its projection onto the $y$-axis.
	
	We claim that if 
	$\text{int}(\pi_2(T))\cap \text{int}(\pi_2(T'))\ne\emptyset$, then $T\sim T'$. To see this, suppose 
	$\text{int}(\pi_2(T))\cap \text{int}(\pi_2(T'))$ is non-empty. The intersection is then an open
	(hence uncountable) set. Let $W$ denote the (countable) set of $y$ coordinates of 
	tops and bottoms of tiles. Thus there exists $y_0\in\text{int}(\pi_2(T))\cap \text{int}(\pi_2(T'))\setminus W$.
	Hence the tiles	on the line $y=y_0$ between $T$ and $T'$ are pairwise adjacent, so that $T$ and $T'$ lie
	in the same equivalence class as required.
	
	Fix a tile $T_0$ in  the tiling $z$ and let $J$ be the convex hull of 
	$\bigcup_{T\sim T_0}\text{int}(\pi_2(T))$.  Then $J$ is an open interval (possibly all of $\R$). 
	We claim that the equivalence class $[T_0]$ of $T_0$ consists precisely of those tiles $T$
	such that $\pi_2(T)\cap J\ne\emptyset$.
	If $T\sim T_0$, then by construction, $\pi_2(T)$ intersects $J$. 
	Conversely if $T$ is a tile such that $\pi_2(T)$ intersects $J$, then the intersection is uncountable
	since $J$ is open, and 
	so we can pick 	$y\in J\cap \pi_2(T)\setminus W$. By definition of $J$, there exists $T_1\in [T_0]$
	whose projection includes points below $y$ and $T_2\in[T_0]$ whose projection includes points above $y$.
	Since $T_1$ and $T_2$ lie in the same equivalence class, 
	there is a sequence of adjacencies connecting $T_1$ to $T_2$. 
	The sequence of adjacencies must include a tile $T'$ whose projection includes $y$. Notice that since 
	$y\not \in W$, $y$ necessarily lies in $\text{int}(\pi_2(T'))$. Now by the previous paragraph, we see that $T\sim T'$.
	Since $T'\sim T_0$, we deduce that $T\sim T_0$ as required.
	
	If $J$ is all of $\mathbb R$, then all tiles have left
	boundaries differing by multiples of $\alpha$. If $J$ is a proper sub-interval $(a,b)$, then the full lines
	$y=a$ and $y=b$ form bi-infinite horizontal shears.
	
	Now let $\mu$ be an invariant probability measure on $Y$ that is ergodic in the direction 
	of the coordinate axes.  By Lemma \ref{lem:nobiinf}, almost every point has no bi-infinite shears, so that
	almost every element of $Y$ has a single equivalence class. 
	Then the collection of elements of $Y$
	whose left endpoints lie in $\alpha\mathbb Z+[0,\alpha/2)$ is invariant under the vertical action, and has
	measure $1/2$, contradicting ergodicity in the directions of the coordinate axes.	
\end{proof}

\subsection{Patterns along shears}
\label{subsec:patterns-on-shears}
For the remainder of the section, we specialize to the case that the rectangular basic tiles $A$ and $B$ are \emph{incommensurable}, meaning that $A$ and $B$ have incommensurable
widths and heights.  

\begin{definition}
Define the \emph{staircase tilings} to be the
collection of tilings using rectangular basic tiles $A$ and $B$ that contain no semi-infinite 
and no bi-infinite shears.  We denote the collection of staircase tilings by $Y_{A,B}^\stair$.
\end{definition}
The motivation for this name will be apparent once we have described these tilings.  

\begin{lem}[ABA Lemma]\label{lem:aba}
	Let $A$ and $B$ be incommensurable and let $x$ be an element of $Y_{A,B}^\stair$.
	Let $\ell$ be a line segment (horizontal or vertical) consisting entirely of tile boundaries and consider the tiles that touch
	$\ell$ from one side. These tiles can have at most one transition between tile types along the segment: 
	from type $A$ to type $B$; or from type $B$ to type $A$.
\end{lem}

\begin{proof}
	Without loss of generality, we consider the tiles lying above a horizontal line segment. Suppose 
	that an element of a tiling contains a block of $A$'s followed by a block of $B$'s followed by a block
	of $A$'s, all of whose bottoms lie along the horizontal line segment. Note that while this segment
	may be a shear, we also include the case that it is not a shear (this is important in Corollary~\ref{cor:endofshear} 
	and Proposition~\ref{prop:ptswshears}). We show inductively that the border between the $A$ and $B$ blocks is forced to propagate
	vertically to infinity, producing a semi-infinite shear and a contradiction of 
$x\in Y_{A,B}^\stair$.

	By hypothesis, there is at least a single row of $B$'s with $A$'s on either side sitting on the line segment.
	Assume there are $m$ consecutive $B$'s in the row.
	Suppose that we have established that there are $n$ consecutive rows of $B$'s one on top of another 
	(the base case $n=1$ being satisfied by assumption). This is illustrated in Figure \ref{fig:aba}.
	Whichever tiles in $x$ are at the left and right ends of the top row of $B$'s must meet the top of the $B$'s in
	the interior of the left and right boundaries of the tile. To see this, note that the vertical
	distance between the top of the $n$th row of $B$ tiles and the $A$ tile at the bottom is $nh_B-h_A$.
	By the incommensurability assumption, this is not a positive integer combination of $h_A$'s and $h_B$'s.  Thus there is no collection of tiles that can be put on top of the $A$ tile so as to have the same top boundary as
	the $n$th row of $B$ tiles. The row of tiles that sits on top of this row is
	then forced to be of width $mw_B$. Using the
	incommensurability again, the only collection of tiles that can exactly fill this in is $m$ $B$'s, so that
	one is forced to have $n+1$ consecutive rows of $B$'s,  completing the induction.
\end{proof}

\begin{figure}
\includegraphics{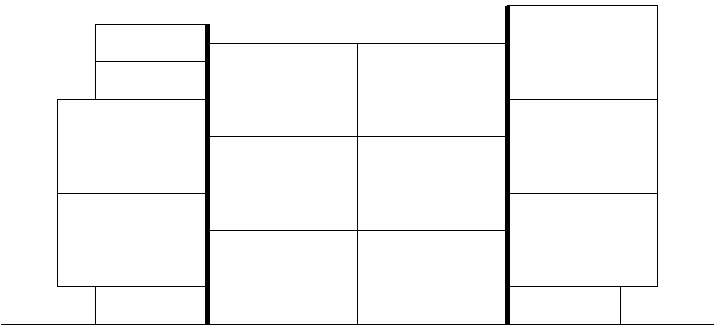}
\caption{Two changes of tile type along a segment generate a semi-infinite shear}
\label{fig:aba}
\end{figure}

\begin{prop}\label{prop:ptswshears}
	Let $A$ and $B$ be incommensurable rectangular basic tiles. If $\mu$ is an
	ergodic invariant measure on $Y_{A,B}$ such that $\mu$-almost every point has bi-infinite horizontal 
	shears, then $\mu$-almost every point consists of complete rows of $A$ tiles and complete rows of $B$
	tiles.
\end{prop}

\begin{proof}
Consider the set of tilings in which a bi-infinite horizontal shear enters 
$[0,1)^2$. This set has positive measure; otherwise
by taking translates we contradict the hypothesis.  Applying the Birkhoff Ergodic 
Theorem, we deduce that $\mu$-almost every
tiling has infinitely many bi-infinite horizontal shears. 

By the ABA Lemma (Lemma~\ref{lem:aba}), in the row immediately above
each horizontal shear, there can be at most one change of tile type.

As in Lemma~\ref{lem:nosemiinf}, the set of tilings in which there is a point in $[0,1)^2$ such that all tiles to the left
are $A$'s and all tiles to the right are $B$'s has measure 0 by the Poincar\'e recurrence 
theorem applied to the horizontal action.
Taking a countable union of translates of this set, we see that in $\mu$-almost every tiling, 
there are no $a\in\R$ such that 
the tiles intersecting the line $y=a$ have exactly one type transition.

In particular, we deduce that in the row immediately above each horizontal shear, there is exactly one tile type. 
The top of this row is another infinite line segment consisting entirely of tile boundaries, so we can apply the
ABA Lemma again to this, and inductively deduce that the tiling consists of full rows of $A$ tiles and full rows of $B$ tiles
only.
\end{proof}

If $\mathcal T=\{A,B\}$ consists of incommensurable rectangular basic tiles, this completes the description of the ergodic invariant probability
measures on $Y_{\mathcal T}$ that have infinite shears. 

\subsection{Shears in staircase tilings}
\label{subsec:shears-in-staircase}

We start with a characterization of the shears:
\begin{lem}[Shear Structure]\label{lem:shearstruc}
	Let $A$ and $B$ be incommensurable rectangular basic tiles and let $x\in Y_{A,B}^\stair$.
	Along each (finite) shear in $x$, there is a non-empty sequence of $A$'s and a 
	non-empty sequence of $B$'s on one side; 
	on the other side the same number of $A$'s and the same number of $B$'s occur in the opposite order.
\end{lem}

\begin{proof}
	Let $x\in Y_{A,B}^\stair$ and consider a shear in $x$, which we assume without
	loss of generality to be horizontal (and finite). By maximality, the continuation of the line segment making 
	up a shear either enters the interior of a tile at the endpoint; or continues as a shared edge (in fact this
	latter possibility is eliminated below). In particular the top side of the shear is completely filled
	out by tiles, as is the bottom side. Since these have exactly the same length, incommensurability
	implies that the top and bottom sides of the shear each have the same number of $A$ tiles and each have 
	the same number of $B$ tiles. Further, by the ABA Lemma, there can be at most one transition on each side 
	between $A$ tiles and $B$ tiles. If the $A$'s and $B$'s occurred in the same order on both sides, then 
	the shear would consist of shared edges. The remaining possibility is that the $A$'s and $B$'s occur in the
	opposite order (see Figure \ref{fig:shear}).
\end{proof}

\begin{figure}
\includegraphics{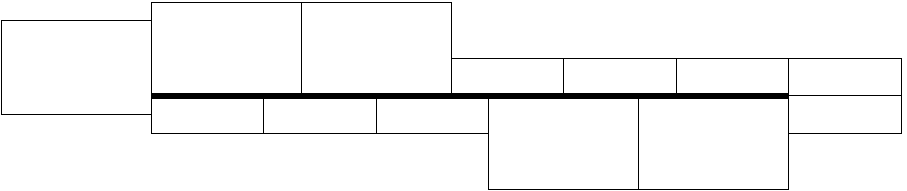}
\caption{Above the shear: 2 $A$'s and 3 $B$'s; Below the shear: 3 $B$'s and 2 $A$'s}
\label{fig:shear}
\end{figure}

\begin{cor}\label{cor:endofshear}
	Let $A$ and $B$ be incommensurable rectangular basic tiles and let $x\in Y_{A,B}^\stair$.
	For each shear in $x$, the continuation of the shear enters the interior of a tile.
	Hence each horizontal shear ends in an interior point of a vertical shear and vice versa.
	
	In particular, shears never cross.
\end{cor}	

\begin{proof}
	It suffices to rule out that the line segment making up the shear is extended by shared edges
	(as occurs at the right of Figure \ref{fig:shear}).
	Suppose without loss of generality that the shear is continued by a shared edge between two $A$'s. 	
	From the Shear Structure Lemma, we see that one side of the shear has $A$'s followed by $B$'s; while the
	other side has $B$'s followed by $A$'s. Along the first side of the shear as extended by the shared
	edges, one sees an $A$ followed by $B$'s followed by
	more $A$'s. This contradicts the ABA Lemma.
\end{proof}

\subsection{Block structures  and basic units in staircase tilings}
\label{subset:block-structures}

We show that the staircases divide a tiling into rectangular patches 
of perfectly aligned $A$ tiles and perfectly aligned $B$ tiles with shears surrounding
them:
\begin{prop}[Rectanglar Blocks]\label{prop:rectangles}
	Let $A$ and $B$ be incommensurable rectangular basic tiles and let $x\in Y_{A,B}^\stair$.
	Then $x$ induces a partition of $\mathbb R^2$ into rectangles.
	Each rectangle is filled either with $A$ tiles or with $B$ tiles. The shears are precisely the boundaries between
	the rectangles.
\end{prop}

\begin{proof}
	Let $x\in Y_{A,B}^\stair$ and consider a shear in $x$, which we assume without
	loss of generality to be horizontal (and finite), and further suppose that the tiles lying above it consist of  $m$ $A$ tiles followed by some sequence of $B$'s.
	We consider the collection of $A$ tiles accessible from these tiles by crossing only shared edges
	and claim they form a rectangular patch. Applying symmetric versions of the argument  completes  the proof of the lemma.
	
The interface between the last $A$ tile and the first $B$ tile forms part of a vertical shear which 
ends where it meets the horizontal shear.
	Since we already know the bottom tile 
	on the left of this shear is an $A$, applying the Shear Structure Lemma, we see that the left side of the vertical
	shear consists of some number $n$ of $A$'s followed by some $B$'s.
	Where the $A$'s transition to $B$'s along the vertical shear is the rightmost
	endpoint of a new horizontal shear going leftwards (see Figure \ref{fig:rect} for an illustration).
	
	Since this shear is exactly $nh_A$ above the lower shear, the space between the shears can only be
	filled with $n$ vertically stacked $A$ tiles. It follows that the transition to $B$'s 
	along the underside of the upper shear
	can only take place after at least $m$ $A$ tiles have been placed. 
	
	By Corollary \ref{cor:endofshear}, we see that a vertical shear goes through the left endpoint
	of the lower horizontal shear. Since a shear cannot be continued 
	by a shared edge, this shear is forced to continue until it meets the
	upper shear. Thus the $m\times n$ block of $A$'s is surrounded by shears on all four sides, as required.	
\end{proof}

\begin{figure}
\includegraphics{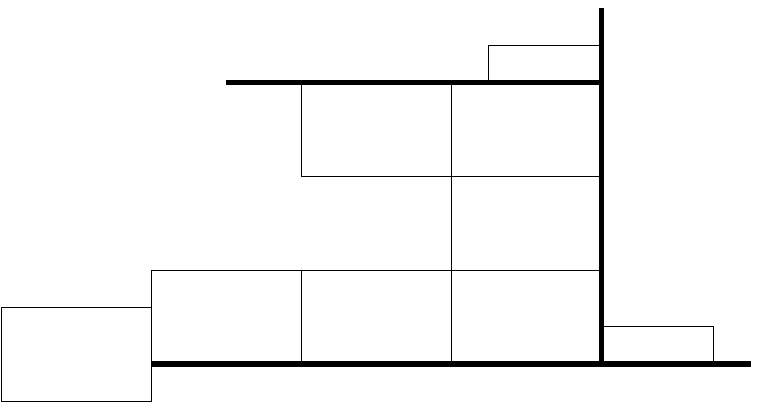}
\caption{The pattern forced by a shear}\label{fig:rect}
\end{figure}

As a consequence of the Rectangular Blocks Proposition, staircase tilings are extremely
rigid. An example of a piece of tiling is illustrated in Figure \ref{fig:rects}.

\begin{figure}
\includegraphics{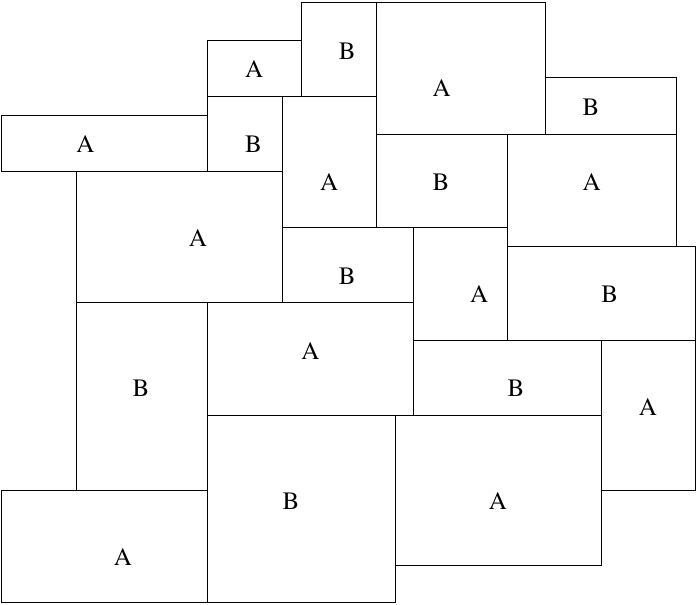}
\caption{Patch of an incommensurable tiling without infinite shears. Note that each rectangle denotes
a block of $A$'s or $B$'s. Notice also that the ABA lemma applies: each shear has $A$'s and $B$'s on one side; 
and the same number of $A$'s and $B$'s in the reverse order on the other side}
\label{fig:rects}
\end{figure}

We call a configuration of a rectangular block of $A$'s and a rectangular block of $B$'s
sitting on top of a horizontal shear a \emph{basic unit}. These come in two types: $AB$ basic units in which 
the $A$ block sits to the left of the $B$ block; and $BA$ basic units in which the $B$ block is to the left.
We show that for tilings in $Y_{A,B}^\stair$, the tiling is completely covered by basic units, disjointly
up to their boundary points, and each basic unit is of the same type. Furthermore, we show that the basic
units have a lattice structure, so that they may be naturally indexed by $\mathbb Z^2$.
This is illustrated in Figure \ref{fig:lattice}.

\begin{figure}
\includegraphics{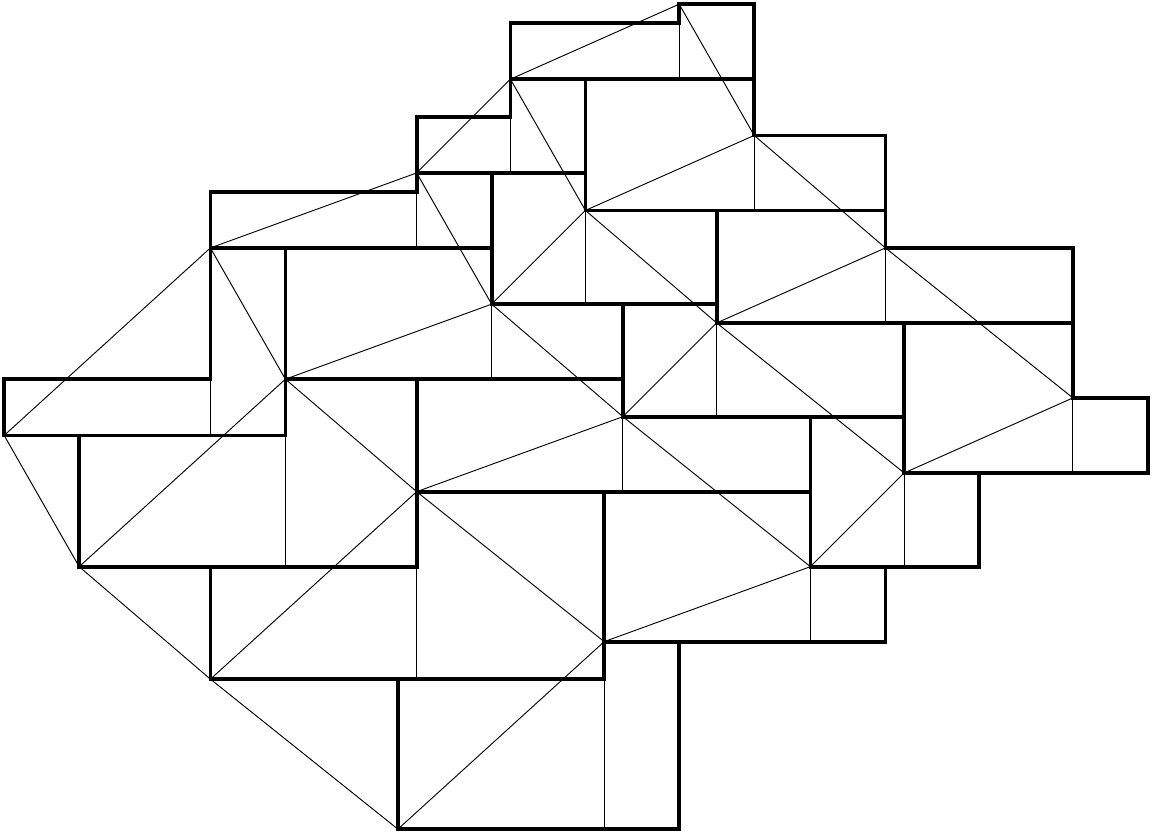}
\caption{Patch of a tiling showing the basic units and their lattice structure (joining the
bottom left corners of the basic units).}
\label{fig:lattice}
\end{figure}

Let $x\in Y_{A,B}^\stair$ and consider a basic unit, which we  assume to be of $AB$ type. 
Notice that the upper left corner of the $B$ block necessarily lies on a vertical shear, and at the left
endpoint of a right-pointing horizontal shear. Since on the underside of this horizontal shear, there is a $B$ block to
the left, it must be the case that above it there is an $A$ block to the left. Thus the upper left corner of the $B$
block is the lower left corner of another $AB$ basic unit and vice versa. Similarly the lower right corner 
of an $AB$ basic unit is the upper right corner of the $A$ block of another $AB$ basic unit and vice versa. 

Given a pair of basic units (of $AB$ type) such that the upper left corner of the $B$ block of the first
is the lower left corner of the second basic unit, we say that the second is the \emph{northeast neighbor} of the first
and conversely, the first is the \emph{southwest neighbor} of the second.
Similarly given a pair of basic units such that the upper right corner of the $A$ block of the first one is the 
lower right corner of the second basic unit, the second is said to be the \emph{northwest neighbor}
of the first; and the first is the \emph{southeast neighbor} of the second. 

We record some basic facts in a lemma.

\begin{lem}[Basic Units]\label{lem:basicunits}
Let $A$ and $B$ be incommensurable and $x\in Y_{A,B}^\stair$. 
Each $AB$ basic unit in $x$ has neighbors in each of the four intermediate directions:
northeast, northwest, southwest and southeast. These neighbors are distinct
and the southwest neighbor of the northeast neighbor is the basic unit itself; similarly
along the southeast--northwest direction.

If a pair of $AB$ basic units are northeast--southwest neighbors, then their
respective $A$ blocks have equal height because they lie on a common vertical shear, while their $B$
blocks have equal width because they lie on a common horizontal shear. 
Similarly if a pair of $AB$ basic units are northwest--southeast neighbors, then their respective $A$ blocks 
have equal width and their $B$ blocks have equal height.
\end{lem}

\begin{proof}
The neighbors are distinct, as the vectors from the lower left corner of a basic
unit to the lower left corners of its northeast, northwest, southwest, and southeast neighbors 
lie in the interiors of the first, second, third and fourth quadrants respectively. 

The fact that northeast neighbor of the southwest neighbor of a basic unit is the basic unit itself
follows from the definition and similarly in the other directions.

The equality of $A$ block heights between northwest--southeast neighbors follows
from the Shear Structure Lemma, since they lie along the same vertical shear. Similar arguments 
establish the equality of $B$ block widths along northwest--southeast neighbors and of
$A$ block widths and $B$ block heights along the northeast--southwest neighbors.
\end{proof}

\subsection{Basic unit lattice structure}
\label{subsec:lattice}

We claim that the basic units are arranged in a lattice that can be 
indexed in a natural way.  We start by showing that 
for each basic unit, the identification of the neighboring units has 
been done in such a way that the directions commute:
\begin{lem}\label{lem:commute}
Let $x\in Y_{A,B}^\stair$, where $A$ and $B$ are incommensurable, 
and suppose that  $x$ contains an $AB$ basic unit.
Then the operations of taking northeast and northwest neighbors commute. 
\end{lem}

\begin{proof}
To see this, let $b_0$ be an $AB$ basic unit in $x$, $b_1$ be its northeast neighbor, and let $b_2$
be the northwest neighbor of $b_0$.

By definition, the bottom left corner of $b_1$ is the upper left corner of the $B$ block in $b_0$ and the
bottom right corner of $b_2$ is the upper right corner of the $A$ block in $b_0$. These two points lie
on the same vertical shear whose top is simultaneously the upper right corner of the $B$ block of $b_2$
and the upper left corner of the $A$ block of $b_1$. This point lies on a horizontal shear whose underside
has $b_2$'s $B$ block to the left and $b_1$'s $A$ block to the right.
The left endpoint of the horizontal shear is the upper left corner of $b_2$'s $B$ block and the 
right endpoint of the shear is the upper right corner of $b_1$'s $A$ block.
On top of this shear, there is an $AB$ basic unit $b_3$, whose lower left corner coincides with the upper left
corner of $b_2$'s $B$ block and whose lower right corner coincides with the upper right corner of $b_1$'s
$A$ block. It follows that $b_3$ is simultaneously the northeast neighbor of $b_1$ and the 
northwest neighbor of $b_2$ as required.
\end{proof}

This gives us a natural labeling of basic units by $\mathbb Z^2$.  Labeling an
arbitrary tile (below, we choose the tile containing the origin) with $(0,0)$, if a tile is labeled $(i,j)$,
we label its northeast neighbor by $(i+1,j)$, its northwest neighbor by $(i,j+1)$, its 
southwest neighbor by $(i-1,j)$ and its southeast neighbor by $(i,j-1)$. By Lemmas \ref{lem:commute}
and \ref{lem:basicunits}, this labeling is consistent. 
The Rectangular Block Proposition 
implies that two basic units are equal or have disjoint interiors.  
We shall show that each basic unit is uniquely labeled, and that the basic units cover the plane.

\begin{lem}\label{lem:uniquelabel}
Let $x\in Y^\stair_{A,B}$ be as above. No basic unit is assigned more than one label by the 
scheme described above.
\end{lem}

\begin{proof}
Given two distinct labels $(i,j)$ and $(i',j')$, without loss assume that $i'\ge i$. If $j'\ge j$, then we see that the
translation vector from the bottom left corner of the $(i,j)$ basic unit to the $(i',j)$ basic unit lies
strictly in the first quadrant, while the translation vector from the bottom left corner of the $(i',j)$ basic
unit to the $(i',j')$ basic unit lies strictly in the second quadrant.  Thus the sum of the translation vectors (that is
the translation vector from the $(i,j)$ basic unit to the $(i',j')$ basic unit) lies strictly in the upper half plane.
It follows that the basic units are distinct. Similarly if $j'\le j$, the translation vector lies strictly in the right
half plane, so that in either case, we see that distinctly labeled basic units are distinct.
\end{proof}

We can now name the staircases in the name of these tilings.  
For $x\in Y_{A,B}^\stair$, let the basic units be labeled as above. 
We define the $j$th \emph{northeast staircase} to be the union of the
bottoms of $A$ tiles and left sides of $B$ tiles in basic units labeled $(i,j)$ for some $i$. We also let
$j$th \emph{expanded northeast staircase} be the union of all of the basic units labeled $(i,j)$ for some $i$.
We make the analogous definitions for northwest staircases: the $i$th northwest 
staircase is the union of the left sides of the
$A$ tiles in basic units labeled $(i,j)$ for some $j$ and tops of the $B$ tiles in basic units labeled $(i-1,j)$ for some $j$.
The $i$th \emph{expanded northwest staircase} is the union of the $A$ tiles in basic units labeled $(i,j)$ for some $j$
and the $B$ tiles in basic units labeled $(i-1,j)$ for some $j$. 
We call such a labeling the \emph{standard labeling of a staircase}.  
This is illustrated in Figure \ref{fig:patchstairs}.

\begin{figure}
\includegraphics{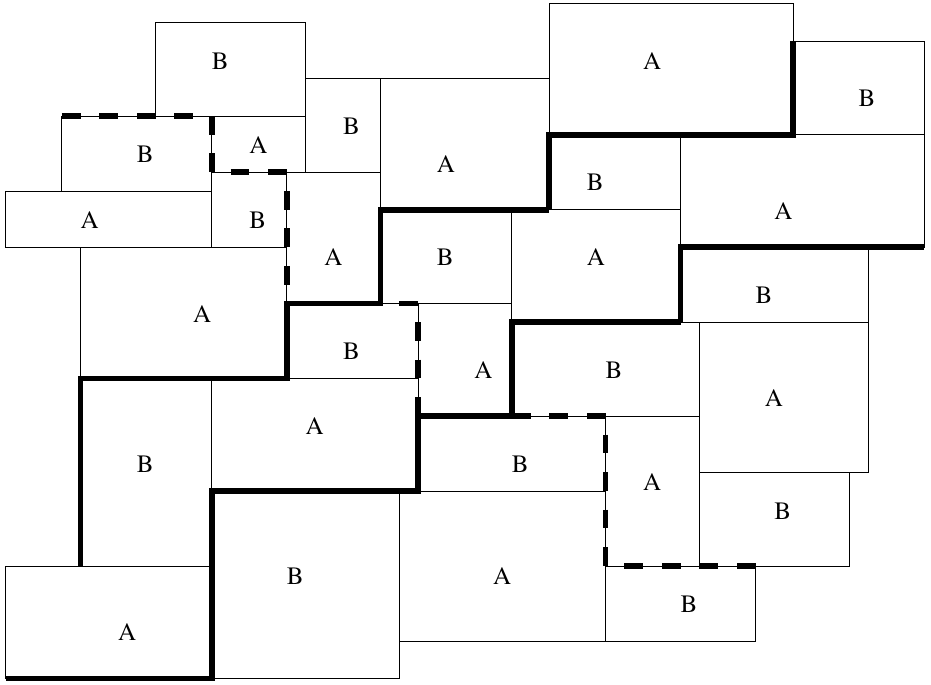}
\caption{A patch of a staircase tiling with two northeast and one northwest staircase indicated.}
\label{fig:patchstairs}
\end{figure}

We summarize the basic properties of staircases:
\begin{lem}\label{lem:staircase}
Let $A$ and $B$ be incommensurable rectangular tiles and $x\in Y_{A,B}^\stair$. 
A staircase is uniquely determined by the bottom left corners of the basic units lying on it. 

The heights of $A$ blocks and the widths of $B$ blocks are constant along any northeast staircase, as are
the widths of $A$ blocks and heights of $B$ blocks along any northwest staircase.

In the standard labelings of the staircases in $x$, 
the $(j+1)$st northeast staircase is a rigid translation of the $j$th staircase through $(-\alpha,\beta)$,
where $\alpha$ is the width of the $B$ tiles in the $(j+1)$st staircase and $\beta$ is the height of the $A$
tiles in the $j$th staircase. The $(i+1)$st northwest staircase is a rigid translation of the $i$th staircase through 
$(\gamma,\delta)$ where $\gamma$ is the width of the $A$ tiles in the $i$th northwest
staircase and $\delta$ is the height
of the $B$ tiles in the $(i+1)$st northwest staircase.
\end{lem}

\begin{proof}
The first statement is clear. The second statement follows from Lemma \ref{lem:basicunits}.

For the last statement we deal with the northeast staircases, the northwest staircases being similar.
The translation vector from the bottom left corner of the $(i,j)$ basic unit
to the bottom left corner of the $(i,j+1)$ basic unit is $(-\alpha,\beta)$, where $\alpha$ is the width of the $B$ block in the 
$(i,j+1)$ basic unit and $\beta$ is the height of the $A$ block in the $(i,j)$ basic unit. By the second statement,
this translation vector is conserved under passing to a northeast or southwest neighbor, so that by the first statement,
the $(j+1)$st northeast staircase is a rigid translation of the $j$th by $(-\alpha,\beta)$.

\end{proof}

We use this to finish the description of the lattice structure of basic units:
\begin{prop}
\label{prop:latticestruct}
Let $A$ and $B$ be incommensurable rectangular tiles and let $x\in Y_{A,B}^\stair$. 
If $x$ contains an $AB$ basic unit, then there is a natural lattice of $AB$
basic units. These basic units cover the plane and have disjoint interiors.
\end{prop}

\begin{proof}
The lattice structure of basic units was described above in Lemmas 
\ref{lem:basicunits}, \ref{lem:commute} and \ref{lem:uniquelabel}.
We are left with showing that the expanded northeast staircases fill out the entire plane.
Let $(-kw_B,\ell h_A)$ denote the translation vector described in Lemma~\ref{lem:staircase}, 
connecting the $j$th northeast staircase and $(j+1)$st.
The area between the $j$th northeast staircase and its translate through $(0,\ell h_A)$ is the union of the $A$ blocks 
belonging to the $j$th expanded staircase. The area between that intermediate translate and its
translate through $(-kw_B,0)$ (that is the $(j+1)$st northeast staircase)
is the union of the $B$ tiles belonging to the expanded staircase of the 
$(j+1)$st staircase. Hence the area between the 0th and $j$th northeast staircases is
contained in the union of $k$th expanded staircases for $k$ running from $0$ to $j$.

For $j>0$, let $E_j$ denote the $j$th expanded northeast staircase, $\bar E_0$
be the 0th expanded northeast staircase together with all points below and $F_j=\bar E_0\cup \bigcup_{0<k\le j}E_k$.

Since the $j$th staircase is a translate of the $(j-1)$st through a vector of the form $(-kw_B,\ell h_A)$, we see
that $F_j$ includes a $\min(w_B,h_A)$-neighborhood of $F_{j-1}$. It follows that all points above the initial
staircase belong to some $F_j$. A similar argument applies below the initial staircase.
\end{proof}

\subsection{Staircase Measure Structure}
\label{subsec:measures-on-staircases}

Given a tiling $x\in Y_{A,B}^\stair$, we make the convention that the $(0,0)$ basic unit is the unit containing the origin
(with the convention that basic units are closed on the left and the bottom and open on the right 
and the top)
and label the remaining units as described in Proposition \ref{prop:latticestruct}. 
Let the $A$ block of the $(i,j)$ basic unit
have dimensions $a_{ij}w_A\times b_{ij}h_A$ and the $B$ block have dimensions $c_{ij}w_B\times d_{ij}h_B$, where 
the $a_{ij},b_{ij},c_{ij},d_{ij}$ are natural numbers. By Lemma \ref{lem:basicunits}, the $a_{ij}$ and $d_{ij}$ 
do not depend on $j$, while the $b_{ij}$ and $c_{ij}$ do not depend on $i$.

\begin{notation}
\label{notation:sequences}
Given a tiling $x\in Y_{A,B}^\stair$, 
we use this to define sequences $(a_i),(b_j),(c_j),$ and $(d_i)$ by setting $a_i=a_{i0}$, $b_j=b_{0j}$, $c_j=c_{0j}$ and $d_i=d_{i0}$.  
\end{notation}
With this convention, we have that the $(i,j)$ basic block consists of a $a_iw_A\times b_jh_A$ $A$ block and a $c_jw_B\times d_ih_B$ $B$
block.

Using this notation, we define a map $\Phi\colon Y_{A,B}^\stair\to (\N\times\N)^\Z\times (\N\times\N)^\Z\times \R^2$
by

\begin{equation}
\label{def:phi}
\Phi(x)=((a_i,d_i)_{i\in\Z},(b_j,c_j)_{j\in\Z},y),
\end{equation}
where $y$ is the position of the origin relative to the bottom left corner of the $(0,0)$ basic unit. 

In fact, we can specify the range of the map more precisely as
\begin{equation}
\label{def:R}
Q=\{((a_i,d_i)_{i\in\Z},(b_j,c_j)_{j\in\Z},y)\colon y\in \tau_{0,0}\}
\end{equation}
where 
\begin{equation}
\label{def:tau00}
\tau_{0,0}=\big([0,a_0w_A)\times [0,b_0h_A)\big)
\times \big([a_0w_A,a_0w_A+c_0w_B)\times [0,d_0h_B)\big)\}.
\end{equation}

\begin{lem}\label{lem:bijection}
The map $\Phi$ defined in Equation~\eqref{def:phi} is a bijection between $Y_{A,B}^\stair$ and $Q$. 
\end{lem}

\begin{proof}
To see that $\Phi$ is injective, notice that 
the principal northeast staircase is determined by the sequence $(a_i,d_i)$. The translation from the
 $j$th northeast staircase
to the $(j+1)$st is $(-c_{j+1}w_B,b_jh_A)$. Knowing the $j$th staircase and the translation vectors
 to the $(j-1)$st and $(j+1)$st
northeast staircases determines the dimensions of the $A$ and $B$ blocks (from the formulae above). 
This means that if two tilings give rise to the same sequences, then they have the same
staircases. If they have the same $y$ value in addition, then the origin is at the same point in the $(0,0)$ tile.

To see that $\Phi$ is surjective, let $r=((a_i,d_i)_{i\in\Z}, (b_j,c_j)_{j\in\Z}, y)\in Q$. 
As described, starting from the sequences
one can build the staircases and the blocks. They always cover all of $\R^2$ as 
previously noted. Hence one can always find
$x\in Y_{A,B}^\stair$ such that $\Phi(x)=r$. 
\end{proof}

We use $\Phi$ to give a complete description of the invariant measures on $Y_{A,B}^\stair$. 
Let $Z=(\N\times\N)^\Z\times (\N\times\N)^\Z$, so that $Q$ as defined in~\eqref{def:R} becomes
$$Q=\{(z,y)\colon z\in Z,y\in \tau_{0,0}(z)\},
$$
where
$\tau_{0,0}$ is 
defined as in~\eqref{def:tau00}.   Write $Z=W_1\times W_2$, where each $W_i$ is a copy of $(\N\times\N)^\Z$.
Let $S_1(w,w')=(\sigma(w),w')$ and $S_2(w,w')=(w,\sigma(w'))$ where $\sigma$ is the left shift. These maps
generate a natural $\Z^2$ action on $Z$.
Since $\Phi$ is a bijection between $Y^\stair_{A,B}$ and $Q$, the
$\R^2$ action on $Y^\stair_{A,B}$ induces an $\R^2$ action on $Q$, preserving the measure $\nu=\Phi^*\mu$.
For $x\in Y^\stair_{A,B}$ and $\vec  v\in\R^2$, $S^{\vec  v}x$ has the same basic units, but relabeled
by adding a constant (the label given to the previous $(0,0)$ block). Of course, the new $y$ coordinate
is the position of the origin in the new $(0,0)$ block after the tiling is shifted. It follows that for $(z,y)\in Q$, 
$S^{\vec  v}(z,y)$ is of the form $(S_1^iS_2^jz,y')$, where the $(i,j)$ is the labeling of the block in $x$
that contains the point with coordinates $\vec v$.

Since the $\R^2$ action on $Q$ locally acts as a translation, we deduce that on the $\R^2$ fibers in $Q$,
the conditional measure of $\nu$ is just Lebesgue measure restricted to $\tau_{0,0}$.
We then define maps between measures on $Q$ and measures on $Z$ and claim that they
establish (up to normalization)
a bijection between the $\R^2$-invariant measures on $Q$ and certain explicitly-described 
$\Z^2$-invariant measures on $Z$.

Given an $\R^2$ invariant measure $\nu$ on $Q$, we define a measure on $Z$ by
\begin{equation}
\label{def:psi1}
\Psi_1(\nu)(C)=\int_{(C\times \R^2)\cap Q} \frac{1}{\text{Area}(\tau_{0,0}(z))}\,d\nu(z,y).
\end{equation}

Conversely, given a $\Z^2$ invariant measure $\lambda$ on $Z$, we define a measure on $Q$
by specifying it on product sets as 
\begin{equation}
\Psi_2(\lambda)(C\times D)=\int_C \text{Area}(\tau_{0,0}(z)\cap D)\,d\lambda(z).\label{eq:psi2}
\end{equation}

\begin{prop}\label{prop:inv}
If $\nu$ is an $\R^2$-invariant probability measure on $Q$, then $\Psi_1(\nu)$ is a finite $\Z^2$ invariant measure on $Z$.
If $\nu$ is ergodic, then so is $\Psi_1(\nu)$.

If $\lambda$ is a $\Z^2$-invariant measure on $Z$ satisfying $\int \text{Area}(\tau_{0,0}(z))\,d\lambda
(z)<\infty$, then $\Psi_2(\lambda)$ is a finite $\R^2$ invariant measure on $Q$. If $\lambda$ is ergodic, so is
$\Psi_2(\lambda)$.

The maps $\Psi_1$ and $\Psi_2$ are inverses of each other.
\end{prop}

\begin{proof}
Since the areas are bounded below in the definition~\eqref{def:psi1},
it follows that $\Psi_1(\nu)$ is finite. 
To show that $\Psi_1(\nu)$ is invariant, let $C\times D$ be a cylinder set in $Z$,
where $C$ and $D$ are cylinder sets specifying 
at least the $0$th and $1$st coordinates of $W_1$ and $W_2$. 
 Maintaining the notation of~\ref{notation:sequences}, note that for all $z\in C\times D$, $\text{Area}(\tau_{0,0}(z))$ is $A_{00}=a_0b_0w_Ah_A+c_0d_0w_Bh_B$, 
the area of the $(0,0)$ basic unit in $z$, while for $z\in S_1(C\times D)$, 
$\text{Area}(\tau_{0,0}(z))$ is $A_{10}=a_1b_0w_Ah_A+c_0d_1w_Bh_B$.
Hence we have that 
$\Psi_1(\nu)(C\times D)=\nu(E)/A_{00}$ and $\Psi_1(\nu)(S_1(C\times D))=\nu(E')/A_{10}$,
where $E=C\times D\times \R^2$ and $E'=\sigma(C)\times D\times\R^2$. However, by considering the 
picture in tiling space, we see that
\begin{equation*}
\frac{m(\{\|v\|\le N\colon S^{\vec  v}(z,y)\in E\})}{A_{00}}
-\frac{m(\{\|v\|\le N\colon S^{\vec  v}(z,y)\in E'\})}{A_{10}}=O(N),
\end{equation*}
where $m$ is Lebesgue measure on $\R^2$. Integrating over $\nu$, exchanging the order of integration,
dividing by $N^2$ and taking the limit,
we see that $\nu(E)/A_{00}=\nu(E')/A_{10}$ and thus $\Psi_1(\nu)$ is $S_1$-invariant. A similar argument
establishes $S_2$-invariance.

To show that $\Psi_2(\lambda)$ is invariant, let $C\times D\subset Q$ where $C$ is a cylinder set
and let $\vec  v\in \R^2$. 
By decomposing $C$ further into cylinder sets if necessary, we may assume that $C$ determines the staircases
out to a radius of at least $\|\vec  v\|+2(w_A+w_B+h_A+h_B)$. Let $z\in C$ and let $y=\Phi^{-1}(z,0)$.
Let $g_1,g_2,\ldots, g_n$ be an enumeration of the basic units in the tiling $z$ that are determined by the
cylinder set $C$. Let $\vec  u_1,\ldots,\vec  u_n$ be the sequence of their bottom left vertices and let
$(i_1,j_1),\ldots,(i_n,j_n)$ be the labels of the basic units as given in Proposition \ref{prop:latticestruct}.

Let $D_k=\{\vec  u\in D\colon \vec  u+\vec  v\in g_k\}$. Then $S^{\vec  v}(C\times D)
=\bigcup_{k=1}^n S_1^{i_k}S_2^{j_k}C\times (D_k+\vec  v-\vec  u_k)$, allowing us to deduce
the invariance of $\Psi_2(\lambda)$ from the shift-invariance of $\nu$ and the translation invariance of
Lebesgue measure.

Now suppose that $\nu$ is ergodic and that $E$ is a $\Z^2$-invariant subset of $Z$. Then 
$(E\times\R^2)\cap Q$ is an $\R^2$-invariant subset of $Q$. It follows that $(E\times\R^2)\cap Q$ has
zero measure or full measure with respect to $\nu$. It follows from the definition~\eqref{def:psi1} of $\Psi_1$ that 
$\Psi_1(\nu)(E)$ is also of zero measure or full measure and so $\Psi_1(\nu)$ is ergodic.

Similarly, suppose $\lambda$ is ergodic and let $F$ be an $\R^2$-invariant subset of $Q$.
Since $F$ is $\R^2$-invariant, we deduce that $F=\{(z,y)\in Q\colon z\in E\}$ for a subset $E$ of $Z$.
By invariance of $F$, it follows that $E$ is $\Z^2$-invariant, and hence has full measure or zero measure.
The same applies to $F$.

Finally, the fact that $\Psi_1(\Psi_2(\lambda))=\lambda$ and $\Psi_2(\Psi_1(\nu))=\nu$ follows in a straightforward manner 
from the definitions.
\end{proof}

We now identify the ergodic $\Z^2$-invariant measures on $Z$.
\begin{lem}\label{lem:Z2inv}
The ergodic $\Z^2$-invariant measures on $Z$ are precisely those measures that are the product of a
pair of $\Z$-invariant ergodic measures on $(\N\times\N)^\Z$.
\end{lem}

\begin{proof}
	Let $\lambda_1$ and $\lambda_2$ be ergodic shift-invariant measures on $W_1$ and $W_2$.
	Then $\lambda=\lambda_1\otimes \lambda_2$, defined on cylinder sets by $\lambda(C\times D)=
	\lambda_1(C)\lambda_2(D)$, is clearly invariant under each of the $S_i$ for $i=1,2$.
	If $A$ is an invariant set under the $\Z^2$-action, then letting $A_w=\{v\in W_2\colon
	(w,v)\in A\}$, we see that $A_w$ is an $S_2$-invariant set, and hence of $\lambda_2$-measure 0 or 1 for 
	$\lambda_1$-almost every $w$. Letting $B=\{w\colon \lambda(A_w)=1\}$, this is an $S_1$-invariant set and hence
	of $\lambda_1$-measure 0 or 1. Thus $\lambda(A)$ is 0 or 1 and so $\lambda$ 
	is ergodic.  
	
	Conversely, let $\lambda$ be an ergodic $\Z^2$-invariant measure on $Z$. Let $\lambda_1$ and $\lambda_2$
	be the (necessarily invariant and ergodic) marginals of $\lambda$. 
	Let $C$ and $D$ be cylinder sets in $W_1$ and $W_2$. Then for $\lambda$-almost every $(w,w')\in Z$, we have
	\begin{align*}
	&\lambda(C\times D)=\lim_{N\to\infty}\frac{1}{N^2}\sum_{0\le i,j<N}\mathbf 1_{C\times D}(\sigma^iw,\sigma^jw')\\
	&=\lim_{N\to\infty}\left(\frac{1}N\sum_{0\le i<N}\mathbf 1_{C}(\sigma^iw)\frac1N
	\sum_{0\le j<N}\mathbf 1_D(\sigma^jw')\right)\\
	&=\lim_{N\to\infty}\frac{1}{N^2}\sum_{0\le i,j<N}\mathbf 1_{C\times W_2}(\sigma^iw,\sigma^jw')
	\lim_{N\to\infty}\frac{1}{N^2}\sum_{0\le i,j<N}\mathbf 1_{W_1\times D}(\sigma^iw,\sigma^jw')\\
	&=\ \lambda_1(C)\lambda_2(D),
	\end{align*}
	and so  $\lambda=\lambda_1\otimes\lambda_2$.
\end{proof}

Maintaining the notation of~\ref{notation:sequences}, this leads to a precise criterion
for an ergodic invariant probability measure on $Z$ 
corresponding to a finite invariant probability measure on $Q$:
\begin{lem}\label{lem:finiteness}
	Let $\lambda=\lambda_1\otimes \lambda_2$ be an ergodic $\Z^2$-invariant probability measure on $Z$.
	Then $\Psi_2(\lambda)$ is a finite measure if and only if $a_0(w)$ and $d_0(w)$ are $\lambda_1$-integrable
	and $b_0(w')$ and $c_0(w')$ are $\lambda_2$-integrable.
\end{lem}

\begin{proof}
	A necessary and sufficient condition for $\Psi_2(\lambda)$ to be a finite measure is $\int
	\text{Area}(\tau_{0,0}(z))\,d\lambda(z)<\infty$. Since 
	$$
	\text{Area}(\tau_{0,0})(w,w')=a_0(w)b_0(w')w_Ah_A
	+c_0(w')d_0(w)w_Bh_B,
	$$
	 we see that
	\begin{align*}
	&\int \text{Area}(\tau_{0,0}(z))\,d\lambda(z)\\
	&=w_Ah_A\int a_0(w)b_0(w')\,d\lambda(w,w')+
	w_Bh_B\int c_0(w')d_0(w)\,d\lambda(w,w')\\
	&=w_Ah_A\int a_0(w)\,d\lambda_1(w)\int b_0(w')\,d\lambda_2(w')\\
	&\ \ +w_Bh_B\int c_0(w')\,d\lambda_2(w')\int d_0(w)\,d\lambda_1(w).
	\end{align*}
	Hence, the $\lambda_1$-integrability
	of $a_0$ and $d_0$ and the $\lambda_2$-integrability of $b_0$ and $c_0$ 
	is a necessary and sufficient condition for finiteness of $\Psi_2(\lambda)$.  
\end{proof}

Combining Proposition~\ref{prop:inv} and Lemmas~\ref{lem:Z2inv} and~\ref{lem:finiteness}, we have shown the structure for staircase measures:
\begin{thm}
\label{thm:staircasemeasurestructure}
Let $A$ and $B$ be incommensurable rectangular basic tiles. The ergodic $\R^2$-invariant
probability measures on $Y_{A,B}^\stair$ are in bijection with the collection of products
$\lambda_1\otimes\lambda_2$ of pairs of ergodic measures on $(\N\times \N)^\Z$ satisfying
integrability of the coordinate functions.
\end{thm}

\subsection{Entropy of invariant measures}
\label{subsec:entropy}

We maintain the notation of Section~\ref{subsec:measures-on-staircases}, and let $H(\mathcal P)$ denote the entropy of a partition $\mathcal P$ and $h_S(\lambda)$ denote the 
measure theoretic entropy of $S$ relative to the measure $\lambda$.  
Recall that $Z=(\N\times\N)^\Z\times (\N\times\N)^\Z$
and that $\tau_{0,0}$ is defined as in~\eqref{def:tau00}.

\begin{lem}\label{lem:finiteentropy}
	Let $\lambda=\lambda_1\otimes \lambda_2$ be an ergodic invariant measure on $Z$ with
	the property that $\int \text{Area}(\tau_{0,0}(z))\,d\lambda(z)<\infty$. 
	Then $h_{S_i}(\lambda_i)<\infty$ for $i=1,2$.
\end{lem}	
	
\begin{proof}
We prove that $h_{S_1}(\lambda_1)<\infty$, the other estimate being identical.
Consider the countable partitions of $W_1$, $\mathcal P_1$ and $\mathcal P_2$ according to
the values taken by $a_0$ and $d_0$. Clearly $\mathcal P_1\vee \mathcal P_2$ is a generating 
partition for $S_1$. We have $h_{S_1}(\lambda_1)\le h_{S_1}(\lambda_1,\mathcal P_1)+
h_{S_1}(\lambda_1,\mathcal P_2)$ and so it suffices to show these terms are finite. We 
consider $h_{S_1}(\lambda_1,\mathcal P_1)$, the other case being similar.

Write $C_n=a_0^{-1}\{n\}$. By Lemma~\ref{lem:finiteness}, we have that $\sum_n n\lambda_1(C_n)<\infty$.
Then 
\begin{align*}
	&h_{S_1}(\lambda_1,\mathcal P_1)\\
	&\le H_{\lambda_1}(\mathcal P_1)
	=-\sum_{n\in\N}\lambda_1(C_n)\log\lambda_1(C_n)\\
	&= -\sum_{\lambda_1(C_n)>e^{-n}}\lambda_1(C_n)\log\lambda_1(C_n)
	-\sum_{\lambda_1(C_n)\le e^{-n}}\lambda_1(C_n)\log\lambda_1(C_n)\\
	&\le \sum_{\lambda_1(C_n)>e^{-n}}n\lambda_1(C_n)+\sum_{ \lambda_1(C_n)\le e^{-n}}ne^{-n}\\
	&\le \sum_{n\in\N}n\lambda_1(C_n)+\sum_{n\in\N}ne^{-n}<\infty,
\end{align*}
where we used for the fourth line the fact that $-x\log x$ is an increasing function on the interval $[0,1/e]$.
\end{proof}

The following is Theorem~\ref{thm:zero} of the introduction stated in the notation of this section.
\begin{thm}\label{thm:zeroentropy}
If $A$ and $B$ are incommensurable rectangular basic tiles, then the topological entropy of the $\R^2$ action on $Y_{A,B}$ is zero.
\end{thm}

\begin{proof}
Using the variational principle, it suffices to show that all of the ergodic invariant measures on $Y_{A,B}$
have zero measure-theoretic entropy. This entropy is, by definition, the measure-theoretic 
entropy of the $\Z^2$-subaction $(S^{\vec v})_{\vec v\in\Z^2}$ on $Y$.
We make the simplifying assumption that all of $w_A$, $h_A$, $w_B$ and $h_B$ are 
greater than 1, so that no unit square can contain the bottom left corner of two tiles. (If this were not satisfied,
we would instead compute the entropy with respect to a finer subaction $(\Z/N)^2$ for an appropriate $N$.)

We define a refining sequence of partitions on $Y_{A,B}$, namely 
$$
\mathcal P^{(n)}=\{A^{(n)}_{i,j}\colon 
0\le i,j<2^n,B^{(n)}_{i,j}\colon 0\le i,j<2^n,C^{(n)}\}, 
$$ 
where $A^{(n)}_{i,j}$ is the set of tilings such that the region
$[i/2^n,(i+1)/2^n)\times [j/2^n,(j+1)/2^n)$ contains the lower left corner of an $A$ tile, 
$B^{(n)}_{i,j}$ is the set of tilings such that the region
$[i/2^n,(i+1)/2^n)\times [j/2^n,(j+1)/2^n)$ contains the lower left corner of a $B$ tile, and $C^{(n)}$ is the event that $[0,1)^2$
does not contain the lower left corner of any tile.

Notice that while $\mathcal P^{(n)}$ is not a generating partition with respect to the $\Z^2$ action, it is the case that 
for any two distinct points $y,y'\in Y_{A,B}$, there exists $\vec  m\in\Z^2$ and $n\in\N$ 
such that $S^{\vec m}y$ and $S^{\vec m}y'$ lie in different elements of $\mathcal P^{(n)}$.
Thus $\bigvee_{n\ge 0}\mathcal P^{(n)}$ agrees up to set of measure $0$ with the Borel $\sigma$-algebra. By standard facts in entropy theory 
(see~\cite{walters} or~\cite{conze}), we obtain 
that $h(\mu)=\lim_{n\to\infty}h(\mu,\mathcal P^{(n)})$ for an $\R^2$-invariant measure $\mu$. 
We therefore show that for any $n\in\N$ and
any $\R^2$-invariant measure $\mu$ on $Y$, $h(\mu,\mathcal P^{(n)})=0$.

We first deal with the case where $\mu$ has infinite shears, restricting to the case that 
$\mu$-almost 
every point has infinite horizontal shears, the case with vertical shears being similar.
By Proposition~\ref{prop:ptswshears}, we have a description of the points on which the measure is supported.

Let $N>0$ be chosen. We compute the
number of non-empty elements in the partition $\mathcal P^{(n)}_{N}=
\bigvee_{0\le i,j<N}T^{-(i,j)}\mathcal P^{(n)}$.

Notice that the element of $\mathcal P^{(n)}_N$ in which a point lies is completely determined by the following
information:
\begin{enumerate}
\item the row types ($A$ or $B$) of the first $N$ rows above the origin;
\item the $y$-coordinates relative to a $2^{-n}$ grid of the bottoms of the rows with $y$ coordinates in the 
range $[0,N)$;
\item the $x$-coordinates relative to a $2^{-n}$ grid of the left edges of all the tiles with lower left corners 
in $[0,N)^2$.
\end{enumerate}

For any point $x\in Y$, each row consists either entirely of $A$ tiles or entirely of $B$ tiles. 
Considering the first $N$ rows whose bottoms are above the origin, we obtain a sequence of $N$
$A$'s and $B$'s corresponding to the order in which they occur. There are  $2^N$ such possibilities.

Given such a sequence, we first address possible positions of the bottoms of the rows relative to the $2^{-n}$
vertical grid implicit in the partition. Notice that there are at most $N$ rows of tiles that could affect the element of $
P^{(n)}_{N}$ in which a point lies. Let $y_0$ be the $y$-coordinate of the lowest row of tiles whose bottom edge
lies above the $x$-axis, so that $0\le y_0<\max(h_A,h_B)$. 
Letting $y_0$ vary, we have that a row bottom crosses the grid at most 
$2^n\max(h_A, h_B)$ times and this can happen in each of the first $N$ rows 
of tiles of the grid.  This leads to a total of 
$2^n\max(h_A, h_B)N$  as an upper bound on the number of possible configurations.  
 

Finally, by a similar argument, in each row whose bottom has a $y$-coordinate in the range $[0,N)$,
the number of possible configurations where the left edges of the tiles lie relative to a $2^{-n}$ grid 
is at most $2^n\max(w_A,w_B)N$.

Multiplying these quantities, we see that
\begin{equation*}\mathcal N(\mathcal P^{(n)}_N)\le2^N2^n\max(h_A,h_B)N(2^n\max(w_A,w_B)N)^N,
\end{equation*}
where $\mathcal N(\mathcal Q)$ denotes the number of non-empty elements of the partition $\mathcal Q$.
In particular, (for fixed $n$) the number of non-empty partition elements
grows slower than $e^{aN^2}$ for any $a>0$.  Thus 
$h(\mu,\mathcal P^{(n)})=0$ for any $n$ and hence
$h(\mu)=0$.

On the other hand, if the ergodic measure $\mu$ is supported on staircase tilings, let $\lambda_1\otimes\lambda_2$ be the corresponding
measure on $W_1\times W_2$. By Theorem \ref{thm:staircasemeasurestructure}, the measures $\lambda_1$ and
$\lambda_2$ are ergodic. By Lemma \ref{lem:finiteentropy}, both of these measures have finite entropy.
Let $\epsilon>0$ be given. There exists $D>0$ such that with probability at least $1-\epsilon$, the $(0,0)$ basic unit
has both dimensions less than $D$. Let $E_1$ be the set where this occurs. 

Let $(w,w')$ denote a point in $W_1\times W_2$ (recall that each $W_i = (\N\times\N)^\Z$). 
Define
$$
\vec u_k=((a_0(w)+a_1(w)+\ldots+a_{k-1}(w))w_A,(d_0(w)+\ldots+d_{k-1}(w))h_B). 
$$
This is the relative position of the $(k,0)$ basic tile to the $(0,0)$ basic tile.
There exist $K>1$, $n_0> 0$, and a subset $E_2$ of $W_1$ of measure at least $1-\epsilon$ such that 
$\|\vec  u_N\| \le KN$ for all $N\ge n_0$. Notice also (by the assumption that $w_A>1$ and $h_B>1$)
that $(\vec  u_N)_i> N$ for $i=1,2$. This defines a segment of the $0$th northeast staircase.
Since the successive northeast staircases lie at least 1 unit above and below their neighbors, we 
deduce that the northeast staircases between the $(-KN)$th and the $(KN)$th cover $[0,N)^2$.
This ensures that $w_0^{N-1}$ and ${w'}_{-KN}^{KN-1}$ determine a patch
of the tiling that includes $[0,N)^2$.

Given this, there are at most $N^2$ tiles in the region, whose boundaries lie on at most $N^2$ $y$-coordinates
and at most $N^2$ $x$-coordinates. The above argument shows that as one moves the origin around 
$\tau_{0,0}(w,w')$, the horizontal boundaries cross the $2^{-n}$ gridlines at most $2^nHN^2$ times and similarly
the vertical boundaries cross the $2^{-n}$ gridlines at most $2^nWN^2$ times, where $W$ and $H$ are the width 
and height of the $(0,0)$ basic unit. 

Let $h_1$ be the entropy of $\lambda_1$ and $h_2$ be the entropy of $\lambda_2$. Then for sufficiently large
$N$, $W_1$ may be covered up to a set of measure $\epsilon$ by $e^{(h_1+\epsilon)N}$ length $N$ cylinder sets
and $W_2$ may be covered up to a set of measure $\epsilon$ by $e^{(h_2+\epsilon)2KN}$ length $2KN$
cylinder sets. Let $E_3$ be the union of the cylinder sets in $W_1$ and $E_4$ be the union of the cylinder sets in
$W_2$.

Hence $E_1\cap (E_2\times W_2)\cap (E_3\times W_2)\cap (W_1\times E_4)$, a set of measure at least $1-4\epsilon$,
is partitioned into at most $e^{(h_1+\epsilon)N}e^{(h_2+\epsilon)(2KN)}(2^nDN^2)^2$ pieces by $\mathcal P^{(n)}_N$.
Notice that this is $O(e^{AN})$ for a suitable $A$, for all sufficiently large $N$.
Hence by standard entropy estimates, we obtain for large $N$
\begin{equation*}
H(\mathcal P^{(n)}_N)\le -(4\epsilon)\log(4\epsilon) + (1-4\epsilon)\log(1-4\epsilon)
+ (4\epsilon)N^2\log|\mathcal P^{(n)}| + AN.
\end{equation*}

Dividing by $N^2$ and taking the limit, we see $h(\mu,\mathcal P^{(n)})\le 4\epsilon\log|\mathcal P^{(n)}|$.
Since $\epsilon$ is arbitrary, $h(\mu,\mathcal P^{(n)})=0$ as claimed.

\end{proof}

Any $\mathbb R^2$ Bernoulli flow is ergodic in every direction and has completely positive entropy and thus Corollary~\ref{c:sharp2} follows immediately from Lemma~\ref{lem:incommens} and Theorem~\ref{thm:zeroentropy}.

\section{Some positive results for $\mathcal T$-tileability with $|\mathcal T|=2$}
\label{sec:positive-2-tiles}

\begin{lem}\label{lem:torustileable}
Let $M\in GL_2(\R)$ be a matrix such that the entries of the first column are rationally independent, as are the 
entries of the second column.
Consider the (non-free) action of $\R^2$ on $\T^2$ defined by
$S^{\vec v}x=(x+M\vec v)\bmod \Z^2$.
Then there exists a collection $\mathcal T$ of two rectangular tiles such that $((S^{\vec v})_{\vec v\in\R^2},\T^2)$
is $\mathcal T$-tileable.
\end{lem}

The condition on the columns of $M$ ensures that the one dimensional actions $(S^{(t,0)})_{t\in\R}$ and
$(S^{(0,t)})_{t\in\R}$ are free.

\begin{proof}
Let $L=M^{-1}$. From the formula for the inverse of a $2\times 2$ matrix, we see the entries of
first row of $L$ are independent, as are the entries of the second row. 
We claim that $L\mathbb Z^2$ has a pair of generators lying in the first and fourth quadrants. To see this, consider the primitive lattice vectors in $L\mathbb Z^2$ lying in a vertical strip $[0,t)\times \R$ and take the two vectors lying closest to the horizontal axis from above and below.

Let these generators be $(a,b)$ and $(c,-d)$. Then let $\mathcal T$ consist the rectangle 
$A$ with $w_A = a$ and $h_A = d$ and the rectangle $B$ with $w_B = c$ and 
$h_B = d$.  
Forming a basic unit by putting the $A$ tile next to the $B$ tile on top of a common
horizontal segment, one can check (as in Section~\ref{subsec:lattice}) that the basic units can be arranged 
to tile $\R^2$ periodically (the displacement vector $(a,b)$ being the translation between consecutive
northwest staircases and the displacement vector $(c,-d)$ being the translation between consecutive
northeast staircases). Let $x\in Y_{A,B}$ be the resulting periodic tiling of the plane.
This is illustrated in Figure \ref{fig:period}. Notice that $T^{\vec u}x=T^{\vec v}x$
if and only if $u-v\in L\Z^2$.

To finish the proof, let $\Phi\colon \T^2\to Y_{A,B}$ be defined by $u\mapsto S^{L\vec u}x$.
The above observation ensures that $\Phi$ is well-defined and it is straightforward to see that 
$\Phi\circ S_{\T^2}^{\vec u}=S_Y^{\vec u}\circ\Phi$, where $S_{\T^2}$ denotes the action on the torus
$S_{\T^2}^{\vec u}(z)=z+M\vec u\bmod \Z^2$ and $S_Y$ denotes the translation action on $Y$.
\end{proof}

\begin{figure}
\includegraphics{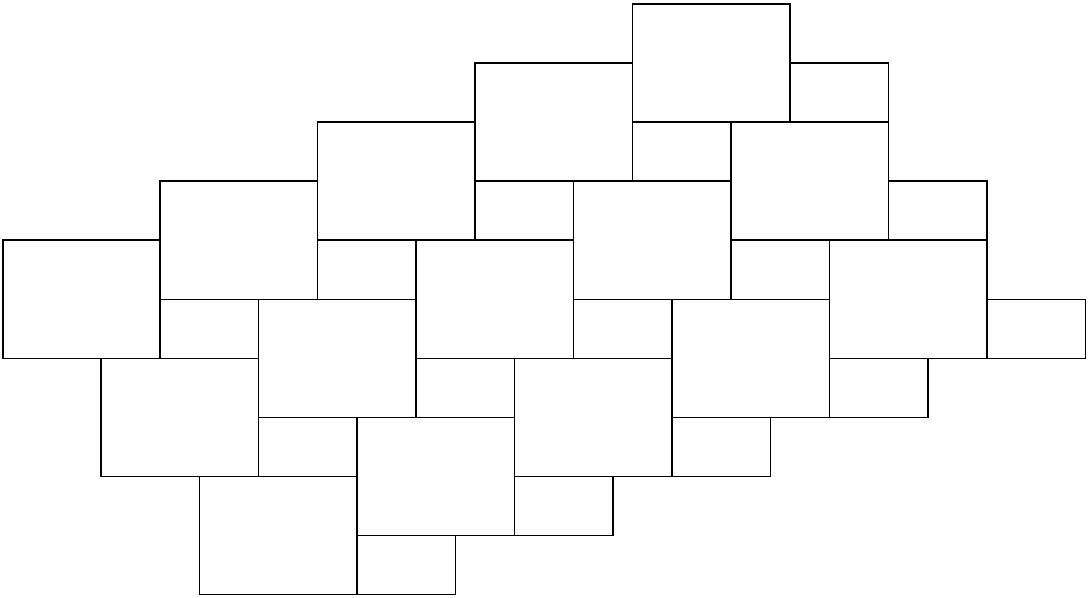}
\caption{A periodic tiling by basic units}\label{fig:period}
\end{figure}

If the matrix $M$ is invertible, but has the property that there is rational dependence
between the entries of one of the columns,
then it turns out that the action is $\mathcal T$-tileable for a set $\mathcal T$ consisting of a single
suitably chosen basic tile. Suppose the first column of $M$ is rationally dependent. Then the second row of $L$ is 
rationally dependent. There exists a matrix $H$ in $SL_2(\Z)$ such that $(LH)_{21}=0$. Note that $L\Z^2=LH\Z^2$.
Let $LH=\begin{pmatrix}a&b\\0&c\end{pmatrix}$, so that the lattice has generators $(a,0)$ and $(b,c)$. We then take 
$\mathcal T$ to consist of an $a\times c$ rectangle, $A$, and take $x$ to be the configuration of rows of $A$'s, each row
shifted rightwards from the row below by $b$. The remainder of the previous argument works as before.

\begin{prop}
Let $(S^{\vec v})_{v\in\R^2}$ be an ergodic measure-preserving action on a space $X$.
Suppose further that there exist linearly independent vectors $\mathbf h_j$ and eigenfunctions
$f_j$ for $j=1,2$ satisfying $f_j(S^{\vec v}x)=\exp(2\pi i\mathbf h_j\cdot \vec v)f_j(x)$.
Then there exists a set $\mathcal T$ consisting of a pair of rectangular tiles such that $(S^{\vec v})
\colon X\to X$ is $\mathcal T$-tileable.
\end{prop}

\begin{proof}
By changing variables, we may regard the $f_j$'s as maps from $X$ to $\mathbb T$ (regarded
as an additive group). Letting $F(x)=(f_1(x),f_2(x))$, we have $F(S^{\vec v}x)=F(x)+M\vec v\bmod \Z^2$,
where $M$ is the matrix with rows $\mathbf h_1$ and $\mathbf h_2$. Hence $F$
is a factor map from $X$ onto the action on the torus appearing in Lemma \ref{lem:torustileable}.
Since that lemma produced a factor map onto the translation action on a tiling space with two rectangular
tile types, composing the factor maps completes the proof.
\end{proof}

\section{Existence of mixing systems that are tileable}
\label{sec:mixing}

\subsection{Statement of the result}

We show:
\begin{thm}
\label{th:mixing}
Assume that $\mathcal T$ consists of two basic tiles $A$ and $B$ whose heights and 
widths are incommensurable and further satisfying $w_A = h_B$ and $w_B=h_A$.  
Then there exists a mixing measure preserving 
system $(X, \mathcal B, \mu, (T_{\vec v})_{\vec v\in\R^2})$ that is $\mathcal T$-tileable.
\end{thm}

The condition $w_A = h_B$ and $w_B=h_A$ is an artifact of the proof that allows us to 
simply the computations.  With significantly more work in the estimates and careful 
choice of bounds in each direction, this condition could be removed.

We produce such a mixing example by showing that there exists a mixing measure on 
$Y_{A,B}^\stair$, and the proof of the mixing property reduces to estimating the distance between 
certain measures on $\R^2$ with respect to suitable metrics. Before turning to the proof of the theorem in Section~\ref{sec:proof-of-mixing}, we describe these metrics.

\subsection{BL and Prokhorov metrics}

\begin{definition}
\label{def:prok}
The {\em Prokhorov metric} is defined on the space of 
Borel probability measures on $\R^d$
by setting $d_\text{Prok}(\eta_1,\eta_2)$ to be 
\begin{equation*}
\inf\{\epsilon\colon \eta_1(C)\le \eta_2(B^\epsilon(C))+
\epsilon,\ \eta_2(C)\le \eta_1(B^\epsilon(C))
+\epsilon\},
\end{equation*}
where $C$ runs over the collection of closed subsets of $\R^d$.
\end{definition}

\begin{definition}
\label{def:BL}
The Bounded Lipschitz (BL) metric is defined on the space of Borel probability measures
on $\R^d$ by setting
$d_\text{BL}(\eta_1,\eta_2)=\sup\{\int f\,d\eta_1-\int f\,d\eta_2\colon\|f\|_\text{Lip}=1\}$,
where the Lipschitz norm of a function is defined to be the sum of its supremum norm
and its minimal Lipschitz constant.
\end{definition}

We use the following inequality relating the Prokhorov and BL metrics:
\begin{thm}[Dudley~\cite{dudley}, Corollary 2]
\label{lem:BLleProk}
If $\eta_1$ and $\eta_2$ are probability measures on a separable metric space, then
$d_\text{BL}(\eta_1,\eta_2) \le 2 d_\text{Prok}(\eta_1,\eta_2)$.
\end{thm}

To estimate the Prokhorov distance, 
we make use of a special case of a theorem of Zaitsev.  Recall that given a probability measure $\eta$ on $\R$, its Fourier transform is given by 
$\widehat{\eta}(t)=\int e^{ixt}\,d\eta(x)$.

\begin{thm}[Zaitsev~\cite{zaitsev}]\label{thm:Zaitsev}
Let $\eta_1$ and $\eta_2$ be measures on $\R$ and let $\Delta(t)$ be the difference between their Fourier transforms. If $\int|x|\,d\eta_1(x)$ and $\int|x|\,d\eta_2(x)$ are finite, then for any $M\ge e$, 
\begin{equation}
\label{eq:Zais}
d_\text{Prok}(\eta_1,\eta_2) \leq 
c_1\left(W^{c_2/\log M}\frac{\log M}M+\left(\int_{-M}^M (\Delta(t)^2+\Delta'(t)^2)\,dt\right)^{1/2}\right),
\end{equation}
where $W=2+\int|x|\,d\eta_1(x)+\int|x|\,d\eta_2(x)$ and $c_1$ and $c_2$ are universal constants.
\end{thm}

\subsection{Proof of Theorem~\ref{th:mixing}}
\label{sec:proof-of-mixing}

Fix incommensurable basic tiles $A$ and $B$.  
Recall that 
we assume that $w_A=h_B$ and $h_A=w_B$ and 
that given a tiling $x\in Y_{A,B}^\stair$, 
we can use it to define 
sequences $(a_i)$, $(b_j)$, $(c_j)$, $(d_i)$ associated 
to it, as in Notation~\ref{notation:sequences}.

Let $\lambda_1$ be the Bernoulli measure on sequences $(a_i,d_i)$
taking values in $\{1,2\}\times\{1,2\}$ with equal probability and let 
$\lambda_2$ be similarly defined on sequences $(b_j, c_j)$.  
Recall that by Lemma~\ref{lem:bijection}, 
the space of tilings is in bijection with $Q$ (see definition~\eqref{def:R}), and furthermore $Q$ is a subset of
$W_1\times W_2\times\R^2$, where each $W_i$ 
is a copy of $(\N\times\N)^\Z$.
Restricting the measure $\lambda_1\times\lambda_2\times\lambda$ to $Q$, 
we obtain a measure $\nu$.  We show that this measure $\nu$ is mixing.

We choose the natural distance function on the space of things defined 
by declaring that two tilings are $\epsilon$-close if they agree up to an 
$\epsilon$-translation on a set of diameter 1/$\epsilon$.
Let $f$ and $g$ be Lipschitz (with respect to this distance on tilings) local functions 
on $Q$ with Lipschitz norm 1, where local means that  the definitions of 
$f$ and $g$ only depend on a fixed neighborhood of the origin.
Thus, the value of $f$ (or $g$) is determined by the finite sequences 
$w_{-K}^K$, ${w'}_{-K}^K$, for some sufficiently large $K$, and by the position in the tile.
Notice also that for any $x\in Q$, $v\mapsto f(S_{\vec v}x)$ is a Lipschitz function
of $\R^2$ with Lipschitz norm 1.

Let $\mathcal F$ denote the smallest $\sigma$-algebra on $Q$ with respect to
which, $(w_1)_{(-(N+K),N+K)^c\cup [-K,K]}$, $(w_2)_{(-(N+K),N+K)^c\cup [-K,K]}$
and $z$, the $\R^2$ component, are measurable (taking $K$ and $N$ 
to be sufficiently large).

Thus,
\begin{align*}
\int f(x)g(S_{\vec v})(x)\,d\nu(x)
&=\int \E_\nu(f\cdot g\circ S_{\vec v}|\mathcal F)(x)\,d\nu(x)\\
&=\int f(x)\E_\nu(g\circ S_{\vec v}|\mathcal F)(x)\,d\nu(x).
\end{align*}
We are left with showing that for sufficiently large $\vec v$, this last integral is 
close to the product $\int f(x)\,d\nu(x)\int g(x)\,d\nu(x)$.  

Let $\Omega$ denote the information not captured 
in $\mathcal F$, meaning that 
$$\Omega=\left(\left(\{1,2\}^2\right)^{[-(N+K),-K)\cup (K,N-K]}\right)^2.
$$ For $\omega=(z,z')\in\Omega$ and $x=(w,w',y)\in Q$,
let $\omega x$ denote the point $(\tilde w,\tilde w',y)$, where
$\tilde w_i=z_i$ if $K<|i|\le N+K$ and $w_i$ otherwise, and 
similarly for $\tilde w'$.
We can then rewrite $\E_\nu(g\circ S_{\vec v}|\mathcal F)(x)$ as 
\begin{align*}
\E_\nu(g\circ S_{\vec v}|\mathcal F)(x)&=
\frac{1}{|\Omega|}\sum_{\omega\in\Omega}g(S_{\vec v}(\omega x)).
\end{align*}

Choose $r_g> 0$ such that if two tilings $x$ and $x'$ agree in 
the ball of radius $r_g$
around the origin, then $g(x)=g(x')$. 
Notice that the tilings $x$ and $\omega x$ agree completely up to translation
off the central \emph{spine} consisting of the northeast and northwest staircases 
with labels $j$ and $i$ satisfying $|i|\le N+K$ and $|j|\le N+K$. The complement of the spine 
consists of four regions (oriented approximately along the four coordinate directions).

We denote the region below the $-(N+K)$ northeast staircase  and above the $(N+K)$ northwest staircase
as region 1; above the $(N+K)$ northeast staircase and the $(N+K)$ northwest staircase as region 2; 
above the $(N+K)$ northeast staircase but below the $-(N+K)$ northwest staircase as region 3; and finally 
below the $-(N+K)$ northeast staircase and the $-(N+K)$ northwest staircase as region 4.
Note that all of these regions depend on choice of the tiling $x$.  

Let $x=(w,w',y)$ be a point in $Q$, where $w=(a_i,d_i)_{i\in\Z}$, 
$w'=(b_j,c_j)_{j\in\Z}$.  Define 
\begin{align*}
\vec U_n(x)&=\begin{cases}((a_0+\ldots+a_{n-1})w_A,
(d_0+\ldots+d_{n-1})h_B)&\text{ for $n\ge 0$};\\
((-a_{-1}-\ldots-a_{n})w_A,(-d_{-1}-\ldots-d_{n})h_B)&\text{ for $n<0$.}
\end{cases}
\\
\vec V_m(x)&=\begin{cases}
(-(c_1+\ldots+c_m)w_B,(b_0+\ldots+b_{m-1})h_A)
&\text{if $m\ge 0$}\\
((c_0+\ldots+c_{m+1})w_B,-(b_{-1}+\ldots+b_m)h_A)&\text{for $m<0$}
\end{cases}
\end{align*}

The vector $\vec U_n(x)$ is the displacement between the bottom left corner of the $(0,0)$
basic unit and the bottom left corner of the $(n,0)$ basic unit. The vector
$\vec V_m(x)$ is the displacement between the bottom left corner of the $(0,0)$ basic unit and the bottom left
corner of the $(0,m)$ basic unit. By the lattice structure, the displacement between the bottom left corner
of the $(0,0)$ basic unit and the bottom left corner of the $(n,m)$ basic unit is $\vec U_n(x)+\vec V_m(x)$.

We further define
\begin{align*}
\vec u_1(x)&=\vec U_{N+K}(x)+\vec V_{-(N+K)}(x)\\
\vec u_2(x)&=\vec U_{N+K}(x)+\vec V_{N+K}(x)\\
\vec u_3(x)&=\vec U_{-(N+K)}(x)+\vec V_{N+K}(x)\\
\vec u_4(x)&=\vec U_{-(N+K)}(x)+\vec V_{-(N+K)}(x).
\end{align*}

As a consequence, provided that $B(\vec v,r_g)$ lies in the same region (say the $i$th) for $x$ and $\omega x$,
we deduce that $g(S_{\vec v} (\omega x))=g(S_{\vec v-\vec h}x)$ where $\vec h=\vec u_i(\omega x)-
\vec u_i(x)$
is the difference in the relative translation vector of the region between $x$ and $\omega x$.

Let $\Omega_0$ be the collection of $x=(w,w',z)$ such that the number of 1's and 2's in each coordinate of $(w,w')$
in each of the ranges $(K,N+K]$ and $[-(N+K),-K)$ is within $\sqrt{2|\log\epsilon|}$ standard deviations of the mean.

Define $E_i$ to be the collection of $x\in\Omega_0$ such that 
$\vec v$ lies in the $i$th region of the
tiling corresponding to $x$, more than $100ND$ away from either of the central
 staircases where $D$ is the maximal
diameter of the basic unit.

Notice that provided $\sqrt{\|\vec v\|}>100ND/\epsilon$, $\vec v$ lies more than $100ND$
away from either of the central staircases with probability $1-O(\epsilon)$. By the 
central limit theorem and standard properties of the normal distribution, $x\in\Omega_0$
with probability $1-O(\epsilon)$. 

For $x\in E_i$, we have that $B(\vec v,r_g)$ also lies in the $i$th region of $\omega x$  and $g(S_{\vec v}(\omega x))
= g(S_{\vec v-(\vec u_i(\omega x)-\vec u_i(x))}x)$ for all $\omega\in\Omega$.

Hence for $x\in E_i$, we have
\begin{align*}
\E_\nu(g\circ S_{\vec v}|\mathcal F)(x)&=
\frac{1}{|\Omega|}\sum_{\omega\in\Omega}g(S_{\vec v+\vec u_i(\omega x)-\vec u_i(x)}x)\\
&=\int g(S_{\vec u}x)\,d\eta(\vec u),
\end{align*}
where $\eta$ is the measure on $\R^2$ given by
\begin{equation*}
\eta=\frac{1}{|\Omega|}\sum_{\omega\in\Omega} \delta_{\vec v+\vec u_i(\omega x)-\vec u_i(x)}.
\end{equation*}

We see that the $x$-coordinates on which $\eta$ is supported are 
just a translation by $(\vec v)_1$
by a weighted sum of binomial random variables. The $y$-coordinates have a 
similar description and are independent
of the $x$-coordinates.

By Lemma~\ref{lem:BLleProk}, it suffices to show that $d_{MT}(\eta,\sigma)<\epsilon$
where $\sigma$ is a measure on $\R^2$ such that for $x$'s belonging to set of measure 
close to 1, one has:
\begin{equation}\label{eq:desiderata}
\left|\int g(S_{\vec v}x)\,d\sigma(\vec v)-\int g(x)\,d\mu(x)\right|
<\epsilon
\end{equation}

We do  this in two steps, first approximating $\eta$ by a normal distribution 
and then approximating the normal distribution
by rectangular pieces with constant density. 

\subsubsection{Approximating $\eta$ by a normal distribution}
Let $\eta_1$ be the projection of $\eta$ to the $x$-coordinate, translated by a constant so that the expectation is
0. Then $\eta_1$ is the distribution of a random variable of the form 
$$
(w_A/2)(\epsilon_1+\ldots+\epsilon_N)
+(w_B/2)(\epsilon_1'+\ldots+\epsilon_N'),
$$
where the $\epsilon_i$ are independent random variables taking the 
values $\pm 1$ with probability $1/2$ each. Let $\zeta$ be the distribution of a normal random variable with mean $0$ and variance ($N/4(w_A^2+w_B^2)$).  Let $N=M^5$, where $M$ 
is to be chosen later,
and note that the quantity $W$ appearing
in~\eqref{eq:Zais} is $O(N)$. This guarantees that the first term of~\eqref{eq:Zais} is $O(\log M/M)$.

Without loss of generality, assume that $w_A < w_B$.  
For the second term in~\eqref{eq:Zais}, 
$\widehat{\eta_1}(t) =(\cos(w_At/2)\cos(w_Bt/2))^N$
and  $\widehat{\zeta}(t) = \exp(-Nt^2(w_A^2+w_B^2)/8)$.  Taking $\Delta(t)$ to be the difference of these quantities and using the Taylor expansion, we have that 
$\Delta(t)$ and $\Delta'(t)$ are $O(N^{-1/2})$  for $|t|\leq 1/w_B$.  Thus for sufficiently 
large $N$, this range gives a trivial contribution to the second term.  
For $|t|>1/w_B$, we have 
that $\widehat{\zeta}(t)$ and $\widehat{\zeta'}(t)$ are $O(Nte^{-aNt^2})$ 
for some constant $a> 0$ that is independent of $N$, and so the integral of their squares contribute
 a total of $O(e^{-aN/w^2_A})$ over the range $|t|> 1/w_B$.
We are left with controlling $\widehat{\eta_1}(t)$ and its derivative in the range $|t|> 1/w_B$.  
Since $\widehat{\eta_1}(t)=(\cos(w_At/2)\cos(w_Bt/2))^N$, the values of $t$ giving
significant
contributions are those for which both $w_At/(2\pi)$ and $w_Bt/(2\pi)$ are close to integers, which depend on the continued
fraction expansion of $w_A/w_B$. Specifically if $p_n/q_n\approx w_A/w_B$, then around $t=2\pi p_n/w_A$,
both terms are close to $\pm 1$. Letting $M=\pi p_n/w_A$, we see that the closest distance between a peak 
of $|\cos(w_At/2)|$ and one of $|\cos(w_Bt/2)|$ is $\Omega(1/M^2)$. 
The contribution when $|\cos(w_At/2)\cos(w_Bt/2)|<0.9$ is clearly negligible. 
One can check that the contribution near a peak is $O(N^{1/2}e^{-ah^2N}+
N^{3/2}h^2e^{-ah^2N})$, where $h$ is the distance between peaks of the cosine functions. 
Thus each peak contributes $O(M^{7/2}e^{-aM})$, and since there are $M$ peaks,
for sufficiently large $M$ (and thus also $N$), the total contribution is arbitrarily small.

Thus by Theorem~\ref{thm:Zaitsev}, for $M$ sufficiently large, 
$\eta_1$ lies close to a normal distribution in the Prokhorov metric, 
meaning that  $d_\text{Prok}(\eta_1,\zeta)$ can be taken to be arbitrarily small.  
By Theorem~\ref{lem:BLleProk}, the same holds for  $d_\text{BL}(\eta_1,\zeta)$.  
Since $\eta_2$ has the same distribution, we also have that 
$d_\text{BL}(\eta_1\times \eta_2,\zeta\times\zeta)$ can be taken to be arbitrarily small.

\subsubsection{Approximating by pieces with constant density}
Choose $L_0$ such that for all $L\ge L_0$, 
$|(1/L^2)\int_{[0,L)^2}g(S_{\vec v}x)\,d\lambda(\vec v)-\int g\,d\nu|<\epsilon^4$
for a set of $x$ belonging to a subset of $Q$ of $\nu$-measure at least $1-\epsilon$ (recall 
that $\lambda$ denotes Lebesgue measure).
Notice that for a (one-dimensional) normal distribution, the part of the space
at least $\sqrt{2|\log\epsilon|}$ standard deviations from the mean has measure $o(\epsilon)$. 
Dividing the part of the range that is at most $\sqrt{2|\log\epsilon|}$ standard 
deviations from the mean into pieces of length $\epsilon/|\log\epsilon|$, on each piece 
in the central region, the ratio of the maximum density to the minimum density is at most $1+\epsilon$.
Hence the distance to a distribution that is piecewise constant on intervals of length $\epsilon/|\log\epsilon|$
times the standard deviation is $O(\epsilon)$. 

Taking $M$ sufficiently large, as above, such that $\sqrt {N(w_A^2+w_B^2)/4}\epsilon/|\log\epsilon|>L_0$, we have that 
$\zeta$ is $\epsilon$-close to a measure $\sigma$ on $\R^2$ that is piecewise constant on
$L\times L$ pieces, where $L=\sqrt{N(w_A^2+w_B^2)/4}\epsilon/|\log\epsilon|$, and for which 
$1-\epsilon$ of the mass is concentrated on $|\log\epsilon|^4/\epsilon^2$ pieces. 
Given $x\in E_i$ and $\vec v$, the probability that all 
$|\log\epsilon|^4/\epsilon^2$ $L\times L$ tiles around
$\vec v$ in $x$ are good is at least $1-\epsilon$, ensuring that $\int g(S_{\vec v-\vec h_i(x)+\vec u}x)\,d\sigma(\vec u)$
is $\epsilon$-close to $\int g\,d\nu$ on a set of large measure.  Thus $\E_\nu(g\circ S_{\vec v})$ is $\epsilon$-close
to $\int g\,d\nu$ on a set of large measure, 
completing the proof of~\eqref{eq:desiderata} and 
the proof of theorem.

\end{document}